\documentclass[reqno]{amsart}
\usepackage[hidelinks]{hyperref}
\usepackage{amssymb}
\usepackage{amsmath}
\usepackage{amsthm}
\usepackage[headings]{fullpage}
\usepackage{bbm}
\newtheorem{theorem}{Theorem}[section]
\newtheorem{lemma}[theorem]{Lemma}
\newtheorem{prop}[theorem]{Proposition}
\theoremstyle{definition}
\newtheorem{defn}[theorem]{Definition}
\newtheorem{remark}[theorem]{Remark}
\numberwithin{equation}{section}
\newcommand{\primecf}{\mathbbm{1}_{\mathbb{P}}}

\begin{document}

\title{Correlations of Almost Primes}
\author{Natalie Evans}
\address{Department of Mathematics, King's College London, Strand, London, WC2R 2LS, UK}
\email{natalie.evans@kcl.ac.uk}
\date{\today}

\maketitle

\begin{abstract}
We prove that analogues of the Hardy-Littlewood generalised twin prime conjecture for almost primes hold on average. Our main theorem establishes an asymptotic formula for the number of integers $n=p_1p_2 \leq X$ such that $n+h$ is a product of exactly two primes which holds for almost all $|h|\leq H$ with $\log^{19+\varepsilon}X\leq H\leq X^{1-\varepsilon}$, under a restriction on the size of one of the prime factors of $n$ and $n+h$. Additionally, we consider correlations $n,n+h$ where $n$ is a prime and $n+h$ has exactly two prime factors, establishing an asymptotic formula which holds for almost all $|h| \leq H$ with $X^{1/6+\varepsilon}\leq H\leq X^{1-\varepsilon}$.
\end{abstract}

\section{Introduction}
\label{sec:introduction}

The generalised twin prime conjecture states that for any integer $k\geq1$ there are infinitely many primes $p$ such that $p+2k$ is also a prime. Hardy and Littlewood \cite{hardylittlewood} conjectured that the number of primes $p\in(X,2X]$ such that $p+2k$ is prime is 
\begin{equation}\label{eq:hardylittlewood}
\sim\frac{\mathfrak{S}(2k)X}{\log^2X}
\end{equation}
as $X\rightarrow\infty$, where $\mathfrak{S}(h)$ is the singular series defined by
\begin{equation}\label{eq:singularseries}
\mathfrak{S}(h):=2\Pi_2\prod_{\substack{p\mid h\\ p>2}}\frac{p-1}{p-2}
\end{equation}
if $h$ is an even integer and zero if $h$ is odd. Here $\Pi_2:=\prod_{p>2}\left(1-\frac{1}{(p-1)^2}\right)$ is the twin prime constant. The Hardy-Littlewood conjecture \eqref{eq:hardylittlewood} is equivalent to showing for any fixed non-zero integer $h$ that
\begin{equation}\label{eq:primecorrelation}
\frac{1}{X}\sum_{X<n\leq2X}\primecf(n)\primecf(n+h)\sim\mathfrak{S}(h)\Bigg(\frac{1}{X}\sum_{X<n\leq2X}\primecf(n)\Bigg)^2,
\end{equation}
where $\primecf$ is the indicator function of the primes, as $X\rightarrow\infty$. While the Hardy-Littlewood conjecture remains wide open it is known to be true on average. In particular, for any fixed $A>0$ we have
\[
\sum_{|h|\le H}\Bigg|\frac{1}{X}\sum_{X<n\leq2X}\primecf(n)\primecf(n+h)-\mathfrak{S}(h)\Bigg(\frac{1}{X}\sum_{X<n\leq2X}\primecf(n)\Bigg)^2\,\Bigg|^2=O\Bigg(\frac{H}{\log^{A+2}X}\Bigg),
\]
so by Chebyshev's inequality we have that \eqref{eq:primecorrelation} holds for all but at most $O_A(H\log^{-A}X)$ values of $|h|\leq H=H(X)$. Mikawa \cite{mikawa} proved that if $X^{1/3+\varepsilon}\leq H\leq X^{1-\varepsilon}$, then for all but at most $O_{\varepsilon,A}(H\log^{-A}X)$ values of $|h|\leq H$ we have that \eqref{eq:primecorrelation} holds. Matom\"aki, Radziwi{\l}{\l} and Tao \cite{mrt}
\footnote{Mikawa proved his result (in his notation) also in the range $X^{1-\varepsilon}\leq H\leq X$. Matom\"aki, Radziwi{\l}{\l} and Tao note that their result can also be proved in this range by their methods. Both results are also proved with a better error term.} 
improved this range, showing that if $0\leq h_0\leq X^{1-\varepsilon}$ and $X^{8/33+\varepsilon}\leq H\leq X^{1-\varepsilon}$ then \eqref{eq:primecorrelation} holds for all but $O_{\varepsilon,A}(H\log^{-A}X)$ values of $h$ such that $|h-h_0|\leq H$.

In this paper we establish an analogue of the Hardy-Littlewood conjecture for integers which have exactly two prime factors (called $E_2$ numbers) which holds on average, provided we restrict the size of one of the prime factors. Given $P>0$ and fixed $\delta>0$ we define $E'_2:=E'_2(P)$ to be the set of integers $n=p_1p_2$ with exactly two prime factors such that $p_1\in(P,P^{1+\delta}]$. The presence of the two prime factors gives the problem a bilinear structure which enables us to go further and we show an asymptotic formula for the correlation 
\[
\frac{1}{X}\sum_{X<n\leq2X}\mathbbm{1}_{E'_2}(n)\mathbbm{1}_{E'_2}(n+h),
\]
where $\mathbbm{1}_{E'_2}$ is the indicator function of the set $E'_2$, which holds for almost all $|h|\leq H$ with $\log^{19+\varepsilon}X\leq H\leq X\log^{-A}X$ and $A>3$.  

\begin{theorem}\label{thm:mainthm}
Let $\varepsilon>0$, $A>3$ be fixed and let $\log^{19+\varepsilon}X\leq H\leq X\log^{-A}X$. Then, there exists some $\eta=\eta(\varepsilon)>0$ such that 
\begin{equation}\label{eq:mainresult}
\frac{1}{X}\sum_{X<n\leq 2X}\mathbbm{1}_{E'_2}(n)\mathbbm{1}_{E'_2}(n+h)\sim\mathfrak{S}(h)\Bigg(\frac{1}{X}\sum_{X<n\leq2X}\mathbbm{1}_{E'_2}(n)\Bigg)^2
\end{equation}
holds for all but at most $O(H\log^{-\eta} X)$ values of $0<|h|\leq H$. Here we define
\[
P:=
\begin{cases}
\log^{17+\varepsilon}X,&\text{ if }\log^{19+\varepsilon}X\leq H\leq \exp\big((\log X)^{\varepsilon^3}\big),\\
\exp\big((\log\log X)^2\big),&\text{ if }\exp\big((\log X)^{\varepsilon^3}\big) < H\leq X\log^{-A}X.
\end{cases}
\]
\end{theorem}

\begin{remark}
The range $X\log^{-A}X\le H\le X$ can also be dealt with by the same methods, see for example \cite{mikawa}, \cite{mrt}. The smallest possible choice of $H$ in the above is $H=\log^{19+\varepsilon}X$, however it may be possible to lower this exponent. In the proof of Theorem \ref{thm:mainthm} we apply the argument of Ter\"av\"ainen \cite[Sections 2-4]{teravainen} showing that almost all intervals $[x,x+\log^{5+\varepsilon}x]$ contain an integer which has exactly two prime factors. The second half of Ter\"av\"ainen's paper is dedicated to lowering the exponent $5+\varepsilon$ to $3.51$ through an argument additionally using some sieve theory and the theory of exponent pairs. We do not apply these ideas here, but it is possible that adapting some aspects of this argument to our proof could lower the exponent of $H$.
\end{remark}

We can prove a similar asymptotic formula for correlations of general $E_2$ numbers which holds on average using the same methods. Making some adjustments to the proof of Theorem \ref{thm:mainthm}, we obtain an asymptotic formula for correlations $n,n+h\in E_2$ which holds for almost all $|h|\leq H$. The cost of considering the set of $E_2$ numbers is taking $H$ larger than in the previous theorem, although we still go beyond what is known for primes.

\begin{theorem}\label{thm:unrestricted}
Let $\varepsilon>0$, $B>0$, $A>3$ be fixed and let $\exp\big((\log X)^{1-\varepsilon}\big)\leq H\leq X\log^{-A}X$. Then, we have that 
\[
\frac{1}{X}\sum_{X<n\leq 2X}\mathbbm{1}_{E_2}(n)\mathbbm{1}_{E_2}(n+h)\sim\mathfrak{S}(h)\Bigg(\frac{1}{X}\sum_{X<n\leq2X}\mathbbm{1}_{E_2}(n)\Bigg)^2
\]
for all but at most $O(H\log^{-B} X)$ values of $0<|h|\leq H$.
\end{theorem}

We can also combine our argument with the work of Mikawa \cite{mikawa} on correlations of primes to study correlations $n, n+h$ where $n$ is a prime and $n+h$ is an $E_2$ number on average. We are still able to take advantage of the bilinear structure provided by the almost prime to go further than what is known for primes and prove an asymptotic formula which holds for almost all $|h|\leq H$ with $H$ as small as $X^{1/6+\varepsilon}$.

\begin{theorem}\label{thm:primealmost}
Let $\varepsilon>0$ be fixed sufficiently small, $B>0$, $A>5$ be fixed and let $X^{1/6+\varepsilon}\leq H\leq X\log^{-A}X$. Then, we have that 
\[
\frac{1}{X}\sum_{X<n\leq2X}\primecf(n)\mathbbm{1}_{E_2}(n+h)\sim\mathfrak{S}(h)\Bigg(\frac{1}{X}\sum_{X<n\leq2X}\primecf(n)\Bigg)\Bigg(\frac{1}{X}\sum_{X<m\leq2X}\mathbbm{1}_{E_2}(m)\Bigg)
\]
for all but at most $O(H\log^{-B} X)$ values of $0<|h|\leq H$.
\end{theorem}

\subsection{Previous Works}

Before outlining the proofs of our results we first discuss some previous results on primes and almost primes which are proved using sieve methods.

Chen's theorem gives that $p+2=q$ such that $p$ is prime and $q$ is either a prime or a product of two primes holds for infinitely many primes $p$. Debouzy \cite{debouzy} proved under the Elliott-Halberstam conjecture that given any $0\leq\beta<\gamma$ there exists $X_0$ such that for all $X\geq X_0$ we have that
\[
\sum_{n\leq X}\Lambda(n)\Lambda(n+2)+\frac{1}{\gamma-\beta}\sum_{n\leq X}\Lambda(n+2)\sum_{\substack{d_1d_2=n\\n^\beta\leq d_1\leq n^\gamma}}\frac{\Lambda(d_1)\Lambda(d_2)}{\log n}=2\Pi_2X(1+o(1)).
\]
This result is proved using an improvement of the Bombieri asymptotic sieve. The Elliott-Halberstam conjecture \cite{elliotthalberstam} (see also \cite{davenport,opera}) concerns the distribution of primes in arithmetic progressions and states that for every $A>0$ and $0<\theta<1$ we have that
\[
\sum_{q\leq x^\theta}\max_{(a,q)=1}\left|\psi(x;q,a)-\frac{x}{\varphi(q)}\right|\ll\frac{x}{\log^Ax},
\]
where we define $\psi(x;q,a):=\sum_{n\leq x , n\equiv a(q)}\Lambda(n)$.

More generally, Bombieri \cite{bombieri} had previously considered pairs $P_k$ and $P_k+2=p$ with $p$ prime and $P_k$ an almost prime with at most $k$ factors. More precisely, defining $\Lambda_k(n):=(\mu\ast\log^k)(n)$ to be the generalised von Mangoldt function where $\ast$ denotes Dirichlet convolution, Bombieri proved that if $k\geq1$ is an integer and $x\geq x_0(k)$ we have
\[
\sum_{n\leq X}\Lambda_k(n)\Lambda(n+2)=2\Pi_2X(\log X)^{k-1}(k+O(k^{4/3}2^{-k/3}))
\]
and, assuming the Elliott-Halberstam conjecture, for $k\geq 2$ we have the asymptotic
\[
\sum_{n\leq X}\Lambda_k(n)\Lambda(n+2)\sim 2\Pi_2kX(\log X)^{k-1}.
\]

There are a number of results regarding bounded gaps between the primes; Zhang \cite{zhang} proved that
\begin{equation}\label{eq:zhangbound}
\liminf_{n\rightarrow\infty}(p_{n+1}-p_n)<7\times 10^7
\end{equation}
and in particular that there exist infinitely many bounded gaps between the primes. Maynard \cite{maynard} improved the above bound to $600$, while the Polymath 8b \cite{polymath} project subsequently improved this to $246$. Under the generalised Elliott-Halberstam conjecture, the best known bound is 6. The twin prime conjecture would amount to proving the above result with the bound $2$. Goldston, Graham, Pintz and Y{\i}ld{\i}r{\i}m \cite{ggpy} proved an almost prime analogue of \eqref{eq:zhangbound}; if $q_1<q_2<\cdots$ denotes the sequence of products of exactly two distinct primes, then 
\[
\liminf_{n\rightarrow\infty}(q_{n+1}-q_n)\leq6.
\]

Integers with exactly two prime factors cannot be counted by sieve methods due to the parity problem - even assuming the Elliott-Halberstam conjecture - and we will instead apply the circle method as in previous works on correlations of primes \cite{mrt}, \cite{mikawa}.

\subsection{Outline of the Proof}

We now discuss the main ideas of the proof of Theorem \ref{thm:mainthm}. We apply the Hardy-Littlewood circle method, first expressing the correlation 
\[
\sum_{X<n\leq2X}\mathbbm{1}_{E'_2}(n)\mathbbm{1}_{E'_2}(n+h)
\]
in terms of the integral
\begin{equation}\label{eq:HLintegral}
\int_0^1\Bigg|\sum_{X<n\leq2X}\mathbbm{1}_{E'_2}(n)e(n\alpha)\Bigg|^2e(-h\alpha)d\alpha.
\end{equation}

We split the integral \eqref{eq:HLintegral} over the unit circle into integrals over the major arcs, the set of points in $(0,1)$ which are well approximated by a rational with a small denominator, i.e. the set of $\alpha\in(0,1)$ such that $|\alpha-a/q|\leq1/(q\log^CX)$ for some integers $(a,q)=1$ with $1\leq q\leq\log^{A'}X$ for some $0<A'<C$, and the minor arcs consisting of the rest of the circle. 

In many problems of this type (see e.g. \cite{mrt}, \cite{mikawa}) where the Hardy-Littlewood circle method is applied, it is usual that the major arcs are treated in a standard way to provide the main term and an error term which is not too difficult to control, while the contribution from the minor arcs is more difficult to bound suitably. Since the correlation
\[
\sum_{X<n\leq2X}\mathbbm{1}_{E'_2}(n)\mathbbm{1}_{E'_2}(n+h)=\sum_{P<p_1,p_3\leq P^{1+\delta}}\sum_{\substack{X<p_1p_2,p_3p_4\leq2X\\p_3p_4=p_1p_2+h}}1
\]
has a bilinear structure, we are in fact able to bound the integral over the minor arcs with relative ease using standard results on bilinear exponential sums. For the major arcs, while we are still able to evaluate the main term in the usual way, the difficulty now lies in estimating the error term.

We will first treat the integral over the minor arcs. We find cancellation in the contribution on average over the shift $h$:
\[
\sum_{0<|h|\leq H}\Bigg|\int_\mathfrak{m}\Bigg|\sum_{X<n\leq2X}\mathbbm{1}_{E'_2}(n)e(n\alpha)\Bigg|^2e(-h\alpha)d\alpha\Bigg|^2.
\]
Expanding the square, applying Poisson summation and Gallagher's Lemma, we can reduce the problem to bounding an integral of the form
\[
\sup_{\alpha\in\mathfrak{m}}\int_X^{2X}\Bigg|\sum_{x<n\leq x+H}\mathbbm{1}_{E'_2}(n)e(n\alpha)\Bigg|^2dx
=\sup_{\alpha\in\mathfrak{m}}\int_X^{2X}\Bigg|\sum_{\substack{x<p_1p_2\leq x+H\\P<p_1\leq P^{1+\delta}}} e(\alpha p_1p_2)\Bigg|^2dx.
\]
The bilinear structure of these sums means we get the required cancellation, as seen in the work of Mikawa \cite{mikawa}. We apply the Cauchy-Schwarz inequality before separating the contributions of the diagonal and off-diagonal terms. The diagonal terms are bounded trivially and a standard argument for bounding bilinear exponential sums is used to bound the off-diagonal terms.  

The major arcs contribute the main term, which is evaluated in a standard way, and an error term. We expand the exponential sum in terms of Dirichlet characters, with a suitable approximation to the principal character providing the main term.

To the remaining terms in the expansion, we again apply Gallagher's Lemma to reduce the problem to understanding almost primes in almost all short intervals. We add and subtract a sum over a longer interval, so that we aim to estimate an expression of the form
\begin{equation}\label{eq:heuristicmajor}
\begin{aligned}
\sum_{q\leq(\log X)^{A'}}\frac{q}{\varphi(q)}\sum_{\substack{\chi(q)\\\chi\neq\chi_0}}\Bigg(&\int_X^{2X}\Bigg|\frac{2}{q(\log X)^C}\sum_{x<n\leq x+\frac{q(\log X)^C}{2}}\mathbbm{1}_{E'_2}(n)\chi(n) -\frac{2}{q\Delta}\sum_{x<n\leq x+\frac{q\Delta}{2}}\mathbbm{1}_{E'_2}(n)\chi(n)\Bigg|^2dx \\
+&\int_X^{2X}\Bigg|\frac{2}{q\Delta}\sum_{x<n\leq x+\frac{q\Delta}{2}}\mathbbm{1}_{E'_2}(n)\chi(n)\Bigg|^2dx\Bigg),
\end{aligned}
\end{equation}
with $\Delta$ slightly smaller than $X$. We are then able to apply Cauchy-Schwarz and what is known about primes in almost all short intervals to the second term. For the estimation of the first term, we adapt the work of Ter\"av\"ainen \cite{teravainen} on almost primes in almost all short intervals (which in turn adapts the work of Matom\"aki and Radziwi{\l}{\l} \cite{mrmult} on multiplicative functions in short intervals). In particular, we first use a Parseval-type bound in order to bound the integral in terms of the mean square of the associated Dirichlet polynomial
\[
\int_{-T}^{T}\Bigg|\sum_{X<n\leq2X}\frac{\mathbbm{1}_{E'_2}(n)\chi(n)}{n^{1+it}}\Bigg|^2dt.
\]
We then factorise this Dirichlet polynomial into a short Dirichlet polynomial corresponding to the smaller prime factor $p_1$ and a longer polynomial corresponding to the larger prime factor $p_2$. The domain of integration is split according to whether the short polynomial is pointwise small. When the shorter polynomial is small, we apply the pointwise bound followed by a mean value theorem. When this shorter polynomial is large, to get sufficient cancellation we further decompose the long Dirichlet polynomial into products of shorter polynomials using Heath-Brown's identity, reducing the problem to estimating type I and type II sums. The type I sums occur when these polynomials are sufficiently long and are in fact partial sums related to Dirichlet $L$-functions. In this case we are able to apply the Cauchy-Schwarz inequality followed by a result on the twisted fourth moment of partial sums of Dirichlet $L$-functions. Otherwise, for the type II sums, we then further split the domain according to whether one of these polynomials is small, in which case it is bounded pointwise before we use a mean value theorem. When the polynomial is large, we apply the Hal{\'a}sz-Montgomery inequality followed by large value theorems.

The proof of Theorem \ref{thm:unrestricted} also follows the argument given above, but we need to make appropriate adjustments to the parameters when applying the circle method and take more care when using the Cauchy-Schwarz inequality. On both the major and minor arcs the application of Cauchy-Schwarz to sums over the smaller prime factor is now too inefficient, but we can overcome this by splitting these sums into dyadic intervals and then combining the contributions. For the proof of Theorem \ref{thm:primealmost}, we combine these ideas for the almost primes with the work of Mikawa \cite{mikawa} on the primes.

Recently, the methods of Matom\"aki and Radziwi{\l}{\l} \cite{mrmult} have been combined with the Hardy-Littlewood circle method to make progress on other problems in analytic number theory. Matom\"aki, Radziwi{\l}{\l} and Tao \cite{mrt2} obtained short averages (of length $\log^B X$ for some large $B>0$) for correlations of divisor functions and the von Mangoldt function, at the cost of weaker error terms. Matom\"aki, Radziwi{\l}{\l} and Tao \cite{mrtchowla} use these ideas to establish that Chowla's conjecture \cite{chowla} holds on average as soon as the length of the average grows with $X$. Recent work of Lichtman and Ter\"av\"ainen  \cite{lichtmanteravainen} shows that a hybrid of Chowla's conjecture and the Hardy-Littlewood conjecture holds on average (see also \cite{lichtman}), with average of length a power of $\log X$.

\subsection{Notation}
Throughout $p,p_i,$ are used to denote prime numbers, while $k,l,m,n,q,r,v$ (with or without subscripts) are positive integers.

As usual, $\mu(\cdot)$ is the M\"obius function and $\varphi(\cdot)$ is the Euler totient function. We let $d_r(n)$ denote the number of solutions to $n=a_1\cdots a_r$ in positive integers. We let $c_q(\cdot)$ be the Ramanujan sum, defined by
\[
c_q(n):=\sum_{\substack{a=1\\(a,q)=1}}^q e\left(\frac{an}{q}\right).
\]
We write $\tau(\cdot)$ for the Gauss sum defined on Dirichlet characters $\chi$ modulo $q$ by
\begin{equation}\label{eq:gauss}
\tau(\chi):=\sum_{n=1}^qe\left(\frac{n}{q}\right)\chi(n),
\end{equation}
which satisfies $\tau(\chi_0)=\mu(q)$.

We use $e:\mathbb{T}\rightarrow\mathbb{R}$ to denote $e(x):=e^{2\pi ix}$, where $\mathbb{T}$ is the unit circle. The notation $\mathbbm{1}_S(\cdot)$ is the indicator function of the set $S$; in particular, we write $\mathbbm{1}_S(n)=1$ if $n\in S$ and $\mathbbm{1}_S(n)=0$ otherwise. Let $\|x\|:=\min_{n\in\mathbb{Z}}|x-n|$ denote distance to the nearest integer. 

We will use $(a,b)$ to denote the greatest common divisor of natural numbers $a$ and $b$, while we write $a\mid b$ if $a$ divides $b$. The shorthand $a\equiv b(q)$ is used to denote that $a$ and $b$ are congruent modulo $q$.

We use the shorthand $\chi(q)$ to denote that the summation is taken over all Dirichlet characters modulo $q$. For complex functions $g_1,g_2$ we use the usual asymptotic notation $g_1(x)=O(g_2(x))$ or $g_1(x)\ll g_2(x)$ to denote that there exist real $x_0$ and $C>0$ such that for every $x\geq x_0$ we have that $|g_1(x)|\leq C|g_2(x)|$. We write $g_1(x)=o(g_2(x))$ if for every $\varepsilon>0$ there exists $x_0$ such that $|g_1(x)|\leq \varepsilon|g_2(x)|$ for all $x\geq x_0$. We use the convention that $\varepsilon>0$ may be different from line to line.

\section{Preliminary Lemmas} 
\label{sec:preliminaries}

We now state several results we will need throughout the argument. We will need the following bound on primes $p$ such that $p+h$ is also prime and the singular series:

\begin{lemma}\label{lem:primepairs}
Let $h\le x$ be an even non-zero integer and suppose that $y\geq4$. The number of primes $p\in(x,x+y]$ such that $p+h$ is also prime is
\[
\ll\frac{\mathfrak{S}(h)y}{(\log y)^2}.
\]
Furthermore, we have that
\[
\sum_{h\leq x}\mathfrak{S}(h)\ll x.
\]
\end{lemma}

\begin{proof}
See \cite[Corollary 3.14]{montgomeryvaughan} and the subsequent exercises.
\end{proof}

We will also need Gallagher's Lemma, which will reduce bounding integrals over the major and minor arcs to studying almost primes in short intervals.

\begin{lemma}[Gallagher's Lemma]\label{lem:gallagherslemma}
 Let $2<y<X/2$. For arbitrary complex numbers $a_n$, we have
\begin{equation}\label{eq:gallagher}
\int_{|\beta|\leq\frac{1}{2y}}\left|\sum_{X<n\leq 2X}a_n e(\beta n)\right|^2d\beta
\ll\frac{1}{y^2}\int_X^{2X}\left|\sum_{x<n\leq x+y}a_n\right|^2dx+y\left(\max_{X<n\leq 2X}|a_n|\right)^2.
\end{equation}
\end{lemma}

\begin{proof}
This Lemma is a modification of \cite[Lemma~1]{gallagher} (see also \cite[Lemma~1]{mikawa}).
\end{proof}

Once we have applied Gallagher's Lemma in the treatment of the major arcs, part of the error term is reduced to a Dirichlet character analogue of a problem on primes in almost all short intervals. We will use the following result adapted from the work of Koukoulopoulos \cite{primesshortap} to bound the second term arising in \eqref{eq:heuristicmajor}:

\begin{lemma}\label{lem:primesalmostall}
Let $A\geq 1$ and $\varepsilon\in(0,\frac{1}{3}]$ be fixed. Let $X\geq1$, $1\leq Q\leq\frac{\Delta}{X^{1/6+\varepsilon}}$ and $\Delta=X^{\theta}$ with $\frac{1}{6}+2\varepsilon\leq\theta\leq 1$. Then we have that
\[
\sum_{q\leq Q}\sum_{\chi(q)}\int_X^{2X}\left|\sum_{x<n\leq x+q\Delta}\left(\Lambda(n)\chi(n)-\delta_\chi\right)\right|^2 dx
\ll\frac{Q^3\Delta^2X}{\log^AX},
\]
where we define $\delta_\chi=1$ if $\chi=\chi_0$ and $\delta_\chi=0$ otherwise.
\end{lemma}

\begin{proof}
The proof can be adapted from the proof given in \cite[Section 4]{primesshortap}. Once the contribution of the imprimitive characters has been bounded in a standard way, the main difference we need to account for when compared with \cite[Theorem 1.2]{primesshortap} is the presence of the square. To account for this, we adjust certain parameters  (namely, replace $D$ with $\sqrt{D}$ and adjust some of the powers of logarithms as needed) and do not require as many applications of the Cauchy-Schwarz inequality. We also note that in our argument we do not need the full strength of \cite[Theorem 1.2]{primesshortap}.
\end{proof}

We use the following Parseval-type result to reduce the problem of finding almost primes in short intervals (cf. the first term of \eqref{eq:heuristicmajor}) to finding cancellation in the mean square of the associated Dirichlet polynomial:

\begin{lemma}[Parseval Bound]\label{lem:dpbound}
Let $a_n$ be arbitrary complex numbers, and let $2\leq h_1\leq h_2\leq\frac{X}{T_0^3}$ with $T_0\geq1$. Define $F(s):=\sum_{X<n\leq2X}\frac{a_n}{n^s}$. Then
\begin{equation}\label{eq:parseval}
\begin{aligned}
&\frac{1}{X}\int_X^{2X}\left|\frac{1}{h_1}\sum_{x<n\leq x+h_1}a_n-\frac{1}{h_2}\sum_{x<n\leq x+h_2}a_n\right|^2dx\\
&\ll\frac{1}{T_0^2}\max_{X<n\leq2X}|a_n|^2+\int_{T_0}^{\frac{X}{h_1}}|F(1+it)|^2dt+\max_{T\geq\frac{X}{h_1}}\frac{X}{Th_1}\int_T^{2T}|F(1+it)|^2dt.
\end{aligned}
\end{equation}
\end{lemma}

\begin{proof}
This is \cite[Lemma 1]{teravainen}, which is a variant of \cite[Lemma~14]{mrmult}.
\end{proof}

Finally, we record an exponential sum bound and a related bound on the sum of reciprocal of the distance to the nearest integer function which provide the necessary cancellation in the estimation of the minor arcs.

\begin{lemma}\label{lem:expsumbound}
Let $\beta\in\mathbb{R}$, then
\[
\sum_{n\leq x}e(\beta n)\ll\min\left(x,\frac{1}{\|\beta\|}\right).
\]
\end{lemma}

\begin{proof}
This is a standard result, see for example \cite[Chapter 8, Eq. (8.6)]{iwanieckowalski}.
\end{proof}

\begin{lemma}\label{lem:distancesum}
If $1<X\leq Y$ and $\alpha\in\mathbb{R}$ satisfies $\alpha=a/q+O(q^{-2})$ with $(a,q)=1$, then we have
\[
\sum_{n\leq X}\min\left(\frac{Y}{n},\frac{1}{\|\alpha n\|}\right)
\ll\left(\frac{Y}{q}+X+q\right)\log(qX).
\]
\end{lemma}

\begin{proof}
This is a standard result, see for example \cite[Chapter 13, Page 346]{iwanieckowalski}.
\end{proof}

\section{Applying the Circle Method}
\label{sec:setup}

To prove Theorem \ref{thm:mainthm}, we will apply the Hardy-Littlewood circle method. Let $\varepsilon>0$ be small and fixed. We define $P>0$ according to the size of $H$ as follows:
\begin{equation}\label{eq:Pdef}
P:=
\begin{cases}
\log^{17+\varepsilon}X&\text{ if }\log^{19+\varepsilon}X\leq H\leq \exp\big((\log X)^{\varepsilon^3}\big),\\ 
\exp\big((\log\log X)^2\big),&\text{ if }\exp\big((\log X)^{\varepsilon^3}\big)< H\leq X\log^{-A}X.
\end{cases}
\end{equation}
It will be more convenient throughout the argument to have a $\log$ weight attached to the indicator function of $E'_2$ as follows:
\begin{defn}\label{def:charfndefn}
We define the arithmetic function $\varpi_2:\mathbb{N}\rightarrow\mathbb{R}$ to be
\[
\varpi_2(n)=
\begin{cases}
\log p_2,&\text{ if }n=p_1p_2\text{ with }P<p_1\leq P^{1+\delta},\\
0,&\text{ otherwise}.
\end{cases}
\]
\end{defn}

From now on we fix $\delta>0$ sufficiently small. We will prove the following asymptotic formula, from which Theorem \ref{thm:mainthm} follows immediately after applying dyadic decomposition:

\begin{theorem}\label{thm:weightedmainthm}
Let $\varepsilon>0$, $A>3$ be fixed and let $\log^{19+\varepsilon}X\leq H\leq X\log^{-A}X$. Then, there exists some $\eta=\eta(\varepsilon)>0$ such that for all but at most $O(H\log^{-\eta} X)$ values of $0<|h|\leq H$ we have that 
\begin{equation}\label{eq:weightedmainresult}
\sum_{X<n\leq 2X}\varpi_2(n)\varpi_2(n+h)=\mathfrak{S}(h)X\left(\sum_{P<p\leq P^{1+\delta}}\frac{1}{p}\right)^2+O\left(\frac{X}{\log^{\eta} X}\right),
\end{equation}
where $\mathfrak{S}(h)$ is the singular series defined in \eqref{eq:singularseries}.
\end{theorem}

\begin{remark}
As $H$ becomes an arbitrarily large power of $\log X$, or is larger than any power of $\log X$, we are able to improve the bound on the error terms to $O(X\log^{-A}X)$ for $A>0$ once we have suitably modified the dependencies between $H$, $P$ and the parameters of the circle method. We also note that, after appropriately modifying the main term, using this result we can in fact prove Theorem \ref{thm:mainthm} with a better error term.
\end{remark}

We consider the integral
\begin{equation}\label{eq:circleint}
\int_{0}^{1}|S(\alpha)|^2e(-h\alpha)d\alpha
=\sum_{X<m,n\leq2X}\varpi_2(m)\varpi_2(n)\int_0^1e(\alpha(m-n-h))d\alpha,
\end{equation}
where for $\alpha\in(0,1)$ we define the exponential sum
\[
S(\alpha):=\sum_{X<n\leq 2X}\varpi_2(n)e(n\alpha).
\]
Then, by the integral identity
\begin{equation}\label{eq:expint}
\int_0^1 e(nx)dx=
\begin{cases}
1,&\text{if }n=0,\\
0,&\text{otherwise},
\end{cases}
\end{equation}
we have that the integral in \eqref{eq:circleint} vanishes unless $m=n+h$. Thus \eqref{eq:circleint} becomes
\begin{align}\label{eq:correlationint}
\int_{0}^{1}|S(\alpha)|^2e(-h\alpha)d\alpha
&=\sum_{X<n\leq2X-h}\varpi_2(n)\varpi_2(n+h)\nonumber\\
&=\sum_{X<n\leq2X}\varpi_2(n)\varpi_2(n+h)+O(h\log^2 X).
\end{align}
This error term will be negligible by our choice of $H$. Thus, except for an acceptable error, we can represent the correlation by an integral over the unit circle.

We split the domain of integration into the major and minor arcs. We define the major arcs $\mathfrak{M}$ to be the set of real $\alpha\in(0,1)$ such that
\begin{equation}\label{eq:majordef}
\left|\alpha-\frac{a}{q}\right|\leq\frac{1}{qQ}\text{ for some }1\leq q\leq Q_0,a<q,(a,q)=1
\end{equation}
with $Q_0:=\log^{A'}X$ and $Q:=P\log X$. Here we define $A'>0$ according to the size of $H$ as follows
\begin{equation}\label{eq:Q0def}
A':=
\begin{cases}
1+\varepsilon^2,&\text{ if }\log^{19+\varepsilon}X\leq H\leq\exp\big((\log X)^{\varepsilon^3}\big),\\
3+\varepsilon^2,&\text{ if }\exp\big((\log X)^{\varepsilon^3}\big)< H\leq X\log^{-A}X.
\end{cases}
\end{equation}
We define the minor arcs $\mathfrak{m}$ to be the rest of the circle, that is, the set of real $\alpha\in(0,1)$ such that 
\begin{equation}\label{eq:minordef}
\left|\alpha-\frac{a}{q}\right|\leq\frac{1}{qQ}\text{ for some }Q_0< q\leq Q,a<q,(a,q)=1.
\end{equation}

\begin{remark}
The parameters satisfy $Q_0<P<Q<H$. Decreasing the size we can take for $P$ would directly reduce how small we are able to take $H$.
\end{remark}

In Section \ref{sec:Minors}, we will prove the following estimate for the integral over the minor arcs:
\begin{prop}[Minor Arc Estimate]\label{prop:minorestimate}
Let $A>3$ be fixed and let $\varepsilon>0$ be fixed  sufficiently small. Let $Q\log^{1+\varepsilon}X\leq H\leq X\log^{-A}X$. With $\mathfrak{m}$ defined as in \eqref{eq:minordef}, for $\alpha\in\mathfrak{m}$ there exists some $\eta=\eta(\varepsilon)>0$ such that
\begin{equation}\label{eq:minorintegral}
\int_{\mathfrak{m}\cap[\alpha-\frac{1}{2H},\alpha+\frac{1}{2H}]}|S(\theta)|^2d\theta
\ll\frac{X}{\log^{1+\eta}X}.
\end{equation}
\end{prop}

Sections \ref{sec:Majors}-\ref{sec:mvdp} will be dedicated to proving the following expression for the integral over the major arcs:
\begin{prop}[Major Arc Integral]\label{prop:majorestimate}
Let $A>3$ be fixed and let $\varepsilon>0$ be fixed sufficiently small. Let $\log^{19+\varepsilon}X\leq H\leq X\log^{-A}X$. With $\mathfrak{M}$ defined as in \eqref{eq:majordef} and $\delta>0$ sufficiently small, there exists some $\eta=\eta(\varepsilon)>0$ such that for all but at most $O(HQ_0^{-1/3})$ values of $0<|h|\leq H$ we have that
\[
\int_\mathfrak{M} |S(\alpha)|^2e(-h\alpha) d\alpha
=\mathfrak{S}(h)X\left(\sum_{P<p\leq P^{1+\delta}}\frac{1}{p}\right)^2
+O\left(\frac{X}{\log^\eta X}\right),
\]
where $\mathfrak{S}(h)$ is the singular series given in \eqref{eq:singularseries}.
\end{prop}

Assuming Proposition \ref{prop:minorestimate} and Proposition \ref{prop:majorestimate}, we can now prove Theorem \ref{thm:weightedmainthm}.

\begin{proof}[Proof of Theorem \ref{thm:weightedmainthm}.]
Following \cite[Section 3]{mrt}, by \cite[Proposition 3.1]{mrt} we have that
\[
\sum_{0<|h|\leq H}\left|\sum_{X<n\leq2X}\varpi_2(n)\varpi_2(n+h)-\int_{\mathfrak{M}}|S(\alpha)|^2e(-h\alpha)d\alpha\right|^2
\ll H\int_\mathfrak{m}|S(\alpha)|^2\int_{\mathfrak{m}\cap[\alpha-\frac{1}{2H},\alpha+\frac{1}{2H}]}|S(\beta)|^2d\beta d\alpha.
\]
By Proposition \ref{prop:minorestimate}, there exists some $\eta=\eta(\varepsilon)>0$ such that
\[
\sup_{\alpha\in\mathfrak{m}}\int_{\mathfrak{m}\cap[\alpha-\frac{1}{2H},\alpha+\frac{1}{2H}]}|S(\beta)|^2d\beta
\ll\frac{X}{\log^{1+\eta}X}.
\]
Noting that by partial summation and Mertens' theorem we have the bound
\[
\int_0^1|S(\alpha)|^2d\alpha
=\sum_{X<n\leq2X}\varpi_2^2(n)
\ll \log X\sum_{X<n\leq2X}\varpi_2(n)
\ll X\log X\sum_{P<p\leq P^{1+\delta}}\frac{1}{p}\ll X\log X
\]
we have that
\[
\sum_{0<|h|\leq H}\left|\sum_{X<n\leq2X}\varpi_2(n)\varpi_2(n+h)-\int_{\mathfrak{M}}|S(\alpha)|^2e(-h\alpha)d\alpha\right|^2\ll\frac{HX^2}{\log^\eta X}.
\]
Applying Chebyshev's inequality and Proposition \ref{prop:majorestimate} then gives the result.
\end{proof}

\section{The Minor Arcs}
\label{sec:Minors}

We first treat the integral over the minor arcs, proving Proposition \ref{prop:minorestimate} by following the proof of \cite[Lemma~8]{mikawa}.

\begin{proof}[Proof of Proposition \ref{prop:minorestimate}.] Starting with the minor arc integral \eqref{eq:minorintegral}, we make the substitution $\theta=\alpha+\beta$ to see that
\[
I:=\int_{\mathfrak{m}\cap[\alpha-\frac{1}{2H},\alpha+\frac{1}{2H}]}|S(\theta)|^2d\theta=\int_{\substack{\alpha+\beta\in\mathfrak{m}\\|\beta|\leq\frac{1}{2H}}}|S(\alpha+\beta)|^2d\beta.
\]
We apply Lemma \ref{lem:gallagherslemma} to the integral to get
\[
I\ll\frac{1}{H^2}\int_{X}^{2X}\left|\sum_{x<n\leq x+H}\varpi_2(n)e(n\alpha)\right|^2dx+H\log^2 X.
\]
The second term is acceptable by our choice of $H$, so it remains to bound the first term. We first consider the case $H\leq \exp((\log X)^{\varepsilon^3})$. We apply the Cauchy-Schwarz inequality to the integrand to get
\begin{equation}\label{eq:minorCS}
\Bigg|\sum_{\substack{x<p_1p_2\leq x+H\\P<p_1\leq P^{1+\delta}}}(\log p_2)e(\alpha p_1p_2)\Bigg|^2 
\leq \Bigg(\sum_{P<m_1\leq P^{1+\delta}}|\primecf(m_1)|^2\Bigg)\Bigg(\sum_{P<m_2\leq P^{1+\delta}}\Bigg|\sum_{x<m_2p\leq x+H}(\log p)e(\alpha m_2 p)\Bigg|^2\Bigg).
\end{equation}
The first term is $\ll \frac{P^{1+\delta}}{\log P}$, while the second term is equal to
\[
\sum_{\substack{x<mp_1,mp_2\leq x+H \\ P<m\leq P^{1+\delta}}} (\log p_1)(\log p_2) e(\alpha m(p_1-p_2)).
\]
Next, we perform the integration on this sum. We may trivially extend the domain of integration to $[0,3X]$ as the integrand is positive. Define the set $\Omega:=\{x:0\leq x\leq 3X, mp_i-H\leq x<mp_i, i=1,2\}$. Exchanging the order of integration and summation and noting that $X<x<mp_1,mp_2\leq x+H\leq3X$, we have that 
\[
I\ll\frac{P^{1+\delta}}{H^2\log P}\sum_{P<m\leq P^{1+\delta}}\left|\sum_{X<mp_1,mp_2\leq3X}(\log p_1)(\log p_2)e(\alpha m(p_1-p_2))\cdot|\Omega|\right|.
\]
If $m|p_1-p_2|>H$, then $|\Omega|=0$. Since we have that $mp_i-H>X-H>0$ and $mp_i\leq3X$ for $i=1,2$, the condition $0\leq x\leq 3X$ is weaker than the condition $\max(mp_1,mp_2)-H\leq x<\min(mp_1,mp_2)$. Therefore, if $m|p_1-p_2|\leq H$ we have that $|\Omega|=H-m|p_1-p_2|$.

We now split the sum into the diagonal terms, $p_1=p_2$, and the off-diagonal terms, $p_1\neq p_2$, denoted by $S_1$ and $S_2$ respectively. The diagonal terms contribute
\begin{equation}\label{eq:minorsdiag}
S_1\ll \frac{P^{1+\delta}}{H\log P}\sum_{P<m\leq P^{1+\delta}}\sum_{\frac{X}{m}<p\leq\frac{3X}{m}}\log^2 p
\ll \frac{XP^{1+\delta}\log(X/P)}{H}.
\end{equation}
Now we bound the off-diagonal terms. Let $r=|p_1-p_2|$. Noting that $0<mr\leq H$, we need to bound
\[
S_2\ll\frac{P^{1+\delta}}{H^2\log P}\sum_{0<r\leq H}\sum_{\substack{\frac{X}{P^{1+\delta}}<p_1,p_2\leq\frac{3X}{P}\\p_2=p_1+r}}(\log p_1)(\log p_2)\Bigg|\sum_{\substack{P<m\leq P^{1+\delta}\\0<m\leq H/r}}e(\alpha mr)(H-mr)\Bigg|.
\]
Noting that $0<m\leq H/r$ and $P<m\leq P^{1+\delta}$, we have that $0<r\leq H/P$. We apply partial summation and Lemma \ref{lem:expsumbound} to the sum over $m$ to see that
\[
S_2\ll \frac{P^{1+\delta}}{H\log P}\sum_{0<r\leq \frac{H}{P}}\min\left(\frac{H}{r},\frac{1}{\|\alpha r\|}\right)\sum_{\substack{\frac{X}{P^{1+\delta}}<p_1,p_2\leq\frac{3X}{P}\\p_2=p_1+r}}(\log p_1)(\log p_2).
\]
By partial summation followed by Lemma \ref{lem:primepairs}, we have that the sum over $p_1,p_2$ is bounded by 
\[
\sum_{\substack{\frac{X}{P^{1+\delta}}<p_1,p_2\leq\frac{3X}{P}\\p_2=p_1+r}}(\log p_1)(\log p_2)
\ll\log^2X\sum_{\substack{\frac{X}{P^{1+\delta}}<p_1,p_2\leq\frac{3X}{P}\\p_2=p_1+r}}1
\ll\frac{\mathfrak{S}(r)X}{P}.
\]
Therefore the contribution of the off-diagonal terms can be bounded by
\[
S_2
\ll\frac{XP^{1+\delta}}{HP\log P}\sum_{0<r\leq \frac{H}{P}}\min\left(\frac{H}{r},\frac{1}{\|\alpha r\|}\right)\mathfrak{S}(r).
\]
We have that $\mathfrak{S}(r)\ll\log\log r$, so applying partial summation we have that
\[
S_2
\ll\frac{XP^{\delta}}{H\log P}\log\log X\sum_{0<r\leq \frac{H}{P}}\min\left(\frac{H}{r},\frac{1}{\|\alpha r\|}\right).
\]
Next, we apply Lemma \ref{lem:distancesum} to the sum over $r$ to get
\begin{equation}\label{eq:minorsoffdiag}
S_2
\ll\frac{XP^\delta }{H}\left(\frac{H}{Q_0}+\frac{H}{P}+Q\right)\log\frac{QH}{P},
\end{equation}
recalling that since $\alpha\in\mathfrak{m}$ we have that $Q_0\leq q\leq Q$. Since $H\leq\exp((\log X)^{\varepsilon^3})$, we note that $\log\frac{QH}{P}\ll(\log X)^{\varepsilon^3}$. Therefore, combining the contributions of the diagonal terms \eqref{eq:minorsdiag} and the off-diagonal terms \eqref{eq:minorsoffdiag}, we find
\[
I
\ll XP^\delta\left((\log X)^{\varepsilon^3}\left(\frac{1}{Q_0}+\frac{1}{P}+\frac{Q}{H}\right)+\frac{P\log(X/P)}{H}\right).
\]
By our choices of $Q_0=\log^{1+\varepsilon^2}X$, $Q\log^{1+\varepsilon} X=P\log^{2+\varepsilon}X\ll H$, we have that
\[
I\ll\frac{X}{\log^{1+\eta}X}
\]
for some $\eta=\eta(\varepsilon)>0$. 

Otherwise, if $H>\exp((\log X)^{\varepsilon^3})$, we split the sum over $P\le p_1\leq P^{1+\delta}$ into dyadic intervals before applying Cauchy-Schwarz in \eqref{eq:minorCS}. We have that $\log\frac{QH}{P}\ll\log X$ and $Q_0=\log^{3+\varepsilon^2}X$, so that the total contribution is
\[
I\ll X\log\log X\left(\log X\left(\frac{1}{Q_0}+\frac{1}{P}+\frac{Q}{H}\right)\right)+\frac{XP\log(X/P)}{H}
\ll\frac{X}{\log^{2+\eta}X}
\]
for some $\eta=\eta(\varepsilon)>0$, which is acceptable.
\end{proof}

\section{The Major Arcs}
\label{sec:Majors}

We now shift our attention to evaluating the contribution of the integral over the major arcs. We will first expand the exponential sum $S(\alpha)$ in terms of Dirichlet characters and suitably approximate the contribution of the principal character, which will provide the main term. We will then evaluate this main term and the sequel will then be dedicated to bounding the error terms that arise from this expansion.

\subsection{Expanding the Exponential Sum}
First, we rewrite the integral over the major arcs by expanding the exponential sum $S(\alpha)$ in terms of Dirichlet characters. Recalling that $\alpha=a/q+\beta$ satisfies \eqref{eq:majordef}, we first define
\begin{align*}
a(\alpha)&:=\frac{\mu(q)}{\varphi(q)}\sum_{P<p\leq P^{1+\delta}}\frac{1}{p}\sum_{X<n\leq2X}e(\beta n),\\
b(\alpha)&:=\frac{1}{\varphi(q)}\sum_{\chi(q)}\tau(\overline{\chi})\chi(a)\sum_{X<n\leq2X}\left(\varpi_2(n)\chi(n)-\delta_\chi\sum_{P_1<p\leq P^{1+\delta}}\frac{1}{p}\right)e(\beta n),\\
A^2(X)&:=\int_{\mathfrak{M}}|a(\alpha)|^2d\alpha,\\
B^2(X)&:=\int_{\mathfrak{M}}|b(\alpha)|^2d\alpha,
\end{align*}
where $\tau(\chi)$ denotes the Gauss sum as defined in \eqref{eq:gauss} and $\delta_\chi=1$ when $\chi=\chi_0$ and is zero otherwise. We will now find the following expression for the integral over the major arcs, once we have expanded the exponential sum:

\begin{lemma}\label{lem:ExpandS}
Let $\mathfrak{M}$ be defined as in \eqref{eq:majordef}. We have that
\[
\int_{\mathfrak{M}}|S(\alpha)|^2 e(-h\alpha)d\alpha=\int_{\mathfrak{M}}|a(\alpha)|^2 e(-h\alpha)d\alpha+O\left(A(X)B(X)+B^2(X)\right).
\]
\end{lemma}

\begin{proof}
Let $\alpha\in\mathfrak{M}$, so that $\alpha=\frac{a}{q}+\beta$ with $q\leq Q_0$, $(a,q)=1$ and $|\beta|\leq\frac{1}{qQ}$. Then
\[
S(\alpha)=\sum_{X<n\leq 2X}\varpi_2(n)e\left(\frac{an}{q}\right)e(\beta n).
\]
By Definition \ref{def:charfndefn}, we have that $n=p_1p_2$ with $P<p_1\leq P^{1+\delta}$. As we have $P>Q_0$, we must have that $(p_1,q)=(p_2,q)=1$ and therefore that $(n,q)=1$. We can now rewrite our expression for $S(\alpha)$ by applying the identity
\begin{equation}\label{eq:addmult}
e\left(\frac{a}{q}\right)=\frac{1}{\varphi(q)}\sum_{\chi(q)}\chi(a)\tau(\overline{\chi})
\end{equation}
which holds for $(a,q)=1$. This gives
\begin{align}
S(\alpha)
&=\frac{1}{\varphi(q)}\sum_{\chi(q)}\tau(\overline{\chi})\chi(a)\sum_{X<n\leq2X}\varpi_2(n)\chi(n)e(\beta n)
\label{eq:Salpha}\\
&=\frac{1}{\varphi(q)}\sum_{\chi(q)}\tau(\overline{\chi})\chi(a)\sum_{\substack{X<p_1p_2\leq2X\\P<p_1\leq P^{1+\delta}}}\chi(p_1)\chi(p_2)(\log p_2)e(\beta p_1p_2)\nonumber,
\end{align}
where we have applied the definition of $\varpi_2$ in the last line. Now we approximate the contribution of the principal character, which will become the main term. First, note that since we have $q\leq Q_0<P<p_1$ we must have that $(p_1p_2,q)=1$ for $X<p_1p_2\leq2X$, so we must have $(\log p_2)\chi_0(p_1)\chi_0(p_2)=\log p_2$ in these ranges. By the prime number theorem, we have that
\[
\sum_{X<n\leq2X}\varpi_2(n)
=\sum_{P<p_1\leq P^{1+\delta}}\sum_{\frac{X}{p_1}<p_2\leq\frac{2X}{p_1}}\log p_2
\sim X\sum_{P<p\leq P^{1+\delta}}\frac{1}{p}.
\]
Therefore we choose to approximate $\sum_{X<n\leq2X}\varpi_2(n)$ by
\[
\sum_{P<p\leq P^{1+\delta}}\frac{1}{p}\sum_{X<n\leq2X}1=X\sum_{P<p\leq P^{1+\delta}}\frac{1}{p}+O(1).
\]
Using this and the fact that $\tau(\chi_0)=\mu(q)$, we approximate the contribution of the principal character to the exponential sum $S(\alpha)$ by
\[
\frac{\mu(q)}{\varphi(q)}\sum_{P<p\leq P^{1+\delta}}\frac{1}{p}\sum_{X<n\leq2X}e(\beta n).
\]
Adding and subtracting this approximation in our expression \eqref{eq:Salpha} for $S(\alpha)$ we have that
\begin{align*}
S(\alpha)
=&\frac{\mu(q)}{\varphi(q)}\sum_{P<p\leq P^{1+\delta}}\frac{1}{p}\sum_{X<n\leq2X}e(\beta n)\\
&+\frac{1}{\varphi(q)}\sum_{\chi(q)}\tau(\overline{\chi})\chi(a)\sum_{X<n\leq2X}\left(\varpi_2(n)\chi(n)-\delta_\chi\sum_{P<p\leq P^{1+\delta}}\frac{1}{p}\right)e(\beta n)\\
=&a(\alpha)+b(\alpha).
\end{align*}
Finally, expanding the square and applying the Cauchy-Schwarz inequality, we have that
\begin{align*}
\int_{\mathfrak{M}}|S(\alpha)|^2 e(-h\alpha)d\alpha
&=\int_{\mathfrak{M}}|a(\alpha)+b(\alpha)|^2e(-h\alpha)d\alpha\\
&=\int_{\mathfrak{M}}|a(\alpha)|^2 e(-h\alpha)d\alpha+O\left(A(X)B(X)+B^2(X)\right),
\end{align*}
as required.
\end{proof}

Thus, in order to prove Proposition \ref{prop:majorestimate} we need to evaluate $\int_{\mathfrak{M}}|a(\alpha)|^2 e(-h\alpha)d\alpha$  (which will also provide a bound for $A^2(X)$) and suitably bound $B^2(X)$.

\subsection{Evaluating the Main Term}
\label{sec:MainTerm}

In this section we evaluate the integral $\int_{\mathfrak{M}}|a(\alpha)|^2 e(-h\alpha)d\alpha$, giving the main term of the asymptotic (and a bound for $A^2(X)$):

\begin{prop}\label{prop:mainterm}
Let $\varepsilon>0$ be fixed sufficiently small. Then for all but at most $O(HQ_0^{-1/3})$ values of $0<|h|\leq H$ we have that
\[
\int_{\mathfrak{M}}|a(\alpha)|^2 e(-h\alpha)d\alpha
=\mathfrak{S}(h)X\left(\sum_{P<p\leq P^{1+\delta}}\frac{1}{p}\right)^2+O\left(\frac{X}{\log^\eta X}\right),
\]
for some $\eta=\eta(\varepsilon)>0$, where we define the singular series $\mathfrak{S}(h)$ as in \eqref{eq:singularseries}.
\end{prop}

Before we can prove Proposition \ref{prop:mainterm}, we need an expression involving the singular series $\mathfrak{S}(h)$.

\begin{lemma}[The Singular Series]\label{lem:singularseries}
Let $h$ be a non-zero even integer and $Q_0$ be defined as in \eqref{eq:Q0def}. Then, for all but at most $O(HQ_0^{-1/3})$ values of $0<|h|\leq H$ we have that
\[
\sum_{q\leq Q_0}\frac{\mu^2(q)c_q(-h)}{\varphi^2(q)}
=\mathfrak{S}(h)+O(Q_0^{-1/3}\log H).
\]
\end{lemma}

\begin{proof}
For similar results, see \cite[Page 39]{mrt} and \cite[Page 35]{vaughanhl}. Rewriting the sum over $q$, we have that
\[
\sum_{q\leq Q_0}\frac{\mu^2(q)c_q(-h)}{\varphi^2(q)}=
\left(\sum_{q=1}^\infty-\sum_{q> Q_0}\right)\frac{\mu^2(q)c_q(-h)}{\varphi^2(q)}.
\]
The first term can be seen to be equal to $\mathfrak{S}(h)$ by calculating the Euler product. It remains to bound the tail of the sum. By \cite[Page 35]{vaughanhl}, we have that
\[
\sum_{0<h\leq H}\left|\sum_{q>Q_0}\frac{\mu^2(q)c_q(-h)}{\varphi^2(q)}\right|^2
\ll\frac{H\log^2 H}{Q_0}.
\]
By Chebyshev's inequality, we have for all but at most $O(HQ_0^{-1/3})$ values of $h$ the bound
\[
\sum_{q>Q_0}\frac{\mu^2(q)c_q(-h)}{\varphi^2(q)}
\ll Q_0^{-1/3}\log H,
\]
as claimed.
\end{proof}

We are now able to complete the proof of Proposition \ref{prop:mainterm}.
\begin{proof}[Proof of Proposition \ref{prop:mainterm}.]
Applying the definition of the major arcs \eqref{eq:majordef} and expanding the square, we have that
\begin{align}
\int_{\mathfrak{M}}|a(\alpha)|^2 e(-h\alpha)d\alpha
&=\sum_{q\leq Q_0}\sum_{\substack{1\leq a\leq q\\(a,q)=1}}\int_{|\beta|\leq\frac{1}{qQ}}\Bigg|\frac{\mu(q)}{\varphi(q)}\sum_{P<p\leq P^{1+\delta}}\frac{1}{p}\sum_{X<n\leq2X}e(\beta n)\Bigg|^2e\left(\frac{-ha}{q}-h\beta\right)d\beta\nonumber\\
&=\left(\sum_{P<p\leq P^{1+\delta}}\frac{1}{p}\right)^2\sum_{q\leq Q_0}\frac{\mu^2(q)c_q(-h)}{\varphi^2(q)}\int_{|\beta|\leq\frac{1}{qQ}}\sum_{X<m,n\leq2X}e(\beta(m-n-h))d\beta\nonumber\\
&=\left(\sum_{P<p\leq P^{1+\delta}}\frac{1}{p}\right)^2\sum_{q\leq Q_0}\frac{\mu^2(q)c_q(-h)}{\varphi^2(q)}I_1,
\label{eq:aintegral}
\end{align}
say. We rewrite the integral $I_1$ as
\begin{align*}
I_1&=\left\{\int_0^1-\int_{\frac{1}{qQ}}^{1-\frac{1}{qQ}}\right\}\sum_{X<m,n\leq2X}e(\beta(m-n-h))d\beta\\
&=:I_2-I_3,
\end{align*}
say. To the first term $I_2$, we apply the identity \eqref{eq:expint} to get
\begin{equation}\label{eq:I1bound}
I_2=\sum_{\substack{X<m,n\leq2X\\m=n+h}}1
=X+O(H)
\end{equation}
and by our choice of $H$ the error term is acceptable. Now we bound the integral $I_3$. Note that $\beta$ is never an integer in the domain of integration, so applying Lemma \ref{lem:expsumbound} to the sums over $m$ and $n$ we have that
\[
I_3
=\int_{\frac{1}{qQ}}^{1-\frac{1}{qQ}}\sum_{X<m,n\leq2X}e(\beta(m-n-h))d\beta
\ll\int_{\frac{1}{qQ}}^{1-\frac{1}{qQ}}\frac{1}{\|\beta\|^2}d\beta
\ll qQ.
\]
Therefore, combining this with \eqref{eq:I1bound}, we have that
\[
I_1=X+O\left(qQ+H\right).
\]
We now substitute this expression for $I_1$ into \eqref{eq:aintegral} to get
\[
\int_{\mathfrak{M}}|a(\alpha)|^2 e(-h\alpha)d\alpha=\sum_{q\leq Q_0}\frac{\mu^2(q)c_q(-h)}{\varphi^2(q)}\Bigg(X\Bigg(\sum_{P<p\leq P^{1+\delta}}\frac{1}{p}\Bigg)^2+O(qQ+H)\Bigg).
\]
To complete the proof, it remains to treat the sum over $q$. By Lemma \ref{lem:singularseries} and our definitions of $H$ and $Q_0$, we find immediately that for all but at most $O(HQ_0^{-1/3})$ values of $0<|h|\leq H$ we have that
\[
\int_{\mathfrak{M}}|a(\alpha)|^2 e(-h\alpha)d\alpha
=\mathfrak{S}(h)X\left(\sum_{P<p\leq P^{1+\delta}}\frac{1}{p}\right)^2+O\left(\frac{X}{\log^\eta X}\right),
\]
for some $\eta=\eta(\varepsilon)>0$, as claimed.
\end{proof}

\section{The Error Term of the Major Arcs}
\label{sec:MajorError}

In order to complete the proof of Proposition \ref{prop:majorestimate}, and therefore the proof of Theorem \ref{thm:weightedmainthm}, we need to find sufficient cancellation in the error term $B^2(X)$ arising on the major arcs. In this section we prove the following bound for $B^2(X)$, which immediately completes the proof of Proposition \ref{prop:majorestimate} when combined with Proposition \ref{prop:mainterm}:

\begin{prop}\label{prop:B2bound}
Let $\varepsilon>0$ be fixed sufficiently small, then there exists some $\eta=\eta(\varepsilon)>0$ such that
\[
B^2(X)\ll\frac{X}{\log^{\eta}X}.
\]
\end{prop}

\subsection{Reduction of the problem}
First, using Gallagher's Lemma (Lemma \ref{lem:gallagherslemma}), we will reduce the problem of estimating $B^2(X)$ to understanding almost primes in almost all short intervals. We first define the following: let $\Delta:=X/T_0^3$ with $T_0:=X^{1/100}$ and  
\begin{align}
B_1(X)&:=\sum_{q\leq Q_0}\frac{q}{\varphi(q)}\sum_{\chi(q)}\int_X^{2X}\Bigg|\frac{2}{qQ}\sum_{x<n\leq x+qQ/2}\Bigg(\chi(n)\varpi_2(n)-\delta_\chi\sum_{P<p\leq P^{1+\delta}}\frac{1}{p}\Bigg)\label{eq:B1def}\\
&\qquad\qquad\qquad\qquad\qquad -\frac{2}{q\Delta}\sum_{x<n\leq x+q\Delta/2}\Bigg(\chi(n)\varpi_2(n)-\delta_\chi\sum_{P<p\leq P^{1+\delta}}\frac{1}{p}\Bigg)\Bigg|^2dx,
\nonumber\\
B_2(X)&:=\sum_{q\leq Q_0}\frac{q}{\varphi(q)}\sum_{\chi(q)}\int_X^{2X}\Bigg|\frac{2}{q\Delta}\sum_{x<n\leq x+q\Delta/2}\Bigg(\chi(n)\varpi_2(n)-\delta_\chi\sum_{P<p\leq P^{1+\delta}}\frac{1}{p}\Bigg)\Bigg|^2dx.
\label{eq:B2def}
\end{align}
Now we are able to state a bound for $B^2(X)$ in terms of $B_1(X)$ and $B_2(X)$:

\begin{prop}\label{prop:B2decomposition}
We have that $B^2(X)\ll B_1(X)+B_2(X)$.
\end{prop}

Then, if we can prove that $B_i(X)\ll X\log^{-\eta}X$ for $i=1,2$, we will immediately be able to conclude Proposition \ref{prop:B2bound}.

\begin{proof}
By definition, we have that $B^2(X)$ equals
\[
\sum_{q\leq Q_0}\sum_{\substack{1\leq a\leq q\\(a,q)=1}}\int_{|\beta|\leq\frac{1}{qQ}}\Bigg|\frac{1}{\varphi(q)}\sum_{\chi(q)}\tau(\overline{\chi})\chi(a){\sum_{X<n\leq2X}}\bigg(\chi(n)\varpi_2(n)-\delta_\chi\sum_{P<p\leq P^{1+\delta}}\frac{1}{p}\bigg)e(\beta n)\Bigg|^2d\beta.
\]
Expanding the square, we have that $B^2(X)$ equals
\begin{align*}
&\sum_{q\leq Q_0}\frac{1}{\varphi^2(q)}\sum_{\chi,\chi'(q)}\tau(\overline{\chi})\tau(\chi')\sum_{\substack{1\leq a\leq q\\(a,q)=1}}\chi(a)\overline{\chi'}(a)\sum_{X<m\leq2X}\left(\chi(m)\varpi_2(m)-\delta_\chi\sum_{P<p\leq P^{1+\delta}}\frac{1}{p}\right)\\
&\times\sum_{X<n\leq2X}\left(\overline{\chi'}(n)\varpi_2(n)-\delta_{\chi'}\sum_{P<p\leq P^{1+\delta}}\frac{1}{p}\right)\int_{|\beta|\leq\frac{1}{qQ}}e(\beta (m-n))d\beta.
\end{align*}
Now, using the definition of Dirichlet characters to trivially extend the sum over $a$ to all $1\leq a\leq q$, we may apply the character orthogonality relation
\[
\sum_{a=1}^q\chi(a)\overline{\chi'}(a)=
\begin{cases}
\varphi(q),&\text{ if }\chi=\chi',\\
0,&\text{ if }\chi\neq\chi',
\end{cases}
\]
to see that $B^2(X)$ is
\begin{align*}
&=\sum_{q\leq Q_0}\frac{1}{\varphi(q)}\sum_{\chi(q)}|\tau(\overline{\chi})|^2\int_{|\beta|\leq\frac{1}{qQ}}\left|\sum_{X<n\leq2X}\left(\chi(n)\varpi_2(n)-\delta_\chi\sum_{P<p\leq P^{1+\delta}}\frac{1}{p}\right)e(\beta n)\right|^2d\beta\\
&\ll\sum_{q\leq Q_0}\frac{q}{\varphi(q)}\sum_{\chi(q)}\int_{|\beta|\leq\frac{1}{qQ}}\left|\sum_{X<n\leq2X}\left(\chi(n)\varpi_2(n)-\delta_\chi\sum_{P<p\leq P^{1+\delta}}\frac{1}{p}\right)e(\beta n)\right|^2d\beta,
\end{align*}
where we have used that $\tau(\overline{\chi})\ll q^{1/2}$ in the last line. Now we apply Lemma \ref{lem:gallagherslemma} to the integral term to get that $B^2(X)$ is bounded by
\begin{equation}\label{eq:shortintervalsum}
\begin{aligned}
\ll
&\sum_{q\leq Q_0}\frac{q}{\varphi(q)}\sum_{\chi(q)}\int_X^{2X}\Bigg|\frac{2}{qQ}\sum_{x<n\leq x+qQ/2}\Bigg(\chi(n)\varpi_2(n)-\delta_\chi\sum_{P<p\leq P^{1+\delta}}\frac{1}{p}\Bigg)\Bigg|^2dx\\
&+Q\log^2X\sum_{q\leq Q_0}\sum_{\chi(q)}\frac{q^2}{\varphi(q)}.
\end{aligned}
\end{equation}
The second term contributes 
\[
Q\log^2X\sum_{q\leq Q_0}\sum_{\chi(q)}\frac{q^2}{\varphi(q)}\ll QQ_0^3\log^2X
\]
to $B^2(X)$, which is negligible. Let $\Delta=X/T_0^3$ with $T_0=X^{1/100}$, then we have that $B^2(X)$ is bounded by
\begin{align*}
\ll&
\sum_{q\leq Q_0}\frac{q}{\varphi(q)}\sum_{\chi(q)}\int_X^{2X}\Bigg|\frac{2}{qQ}\sum_{x<n\leq x+qQ/2}\Bigg(\chi(n)\varpi_2(n)-\delta_\chi\sum_{P<p\leq P^{1+\delta}}\frac{1}{p}\Bigg)\\
&\qquad\qquad\qquad\qquad\quad -\frac{2}{q\Delta}\sum_{x<n\leq x+q\Delta/2}\Bigg(\chi(n)\varpi_2(n)-\delta_\chi\sum_{P<p\leq P^{1+\delta}}\frac{1}{p}\Bigg)\Bigg|^2dx\\
&+\sum_{q\leq Q_0}\frac{q}{\varphi(q)}\sum_{\chi(q)}\int_X^{2X}\Bigg|\frac{2}{q\Delta}\sum_{x<n\leq x+q\Delta/2}\Bigg(\chi(n)\varpi_2(n)-\delta_\chi\sum_{P<p\leq P^{1+\delta}}\frac{1}{p}\Bigg)\Bigg|^2dx\\
=&B_1(X)+B_2(X),
\end{align*}
as claimed.
\end{proof}

\subsection{Bounding \texorpdfstring{$B_2(X)$}{B2(X)}} 
First, we prove the the following estimate for $B_2(X)$, which will be reduced to a Dirichlet character analogue of a problem on primes in almost all short intervals.

\begin{prop}\label{prop:B2Xbound} 
Let $C>0$ be fixed, then with $B_2(X)$ as defined in \eqref{eq:B2def} we have 
\[
B_2(X)\ll\frac{X}{\log^C X}.
\]
\end{prop}

\begin{proof}
We separate the cases $\chi=\chi_0$ and $\chi\neq\chi_0$.
If $\chi=\chi_0$, we have that
\[
\frac{2}{q\Delta}\sum_{x<n\leq x+q\Delta/2}\left(\varpi_2(n)-\sum_{P<p\leq P^{1+\delta}}\frac{1}{p}\right)
=\frac{2}{q\Delta}\sum_{P<p_1\leq P^{1+\delta}}\sum_{\frac{x}{p_1}<p_2\leq\frac{x+q\Delta/2}{p_1}}\log p_2-\sum_{P<p\leq P^{1+\delta}}\frac{1}{p}+O\left(\frac{1}{q\Delta}\right).
\]
We now apply the prime number theorem in short intervals (see, for example, \cite[Chapter 10.5]{iwanieckowalski}), finding that
\[
\frac{2}{q\Delta}\sum_{P<p_1\leq P^{1+\delta}}\sum_{\frac{x}{p_1}<p_2\leq\frac{x+q\Delta/2}{p_1}}\log p_2
=\sum_{P<p\leq P^{1+\delta}}\frac{1}{p}+O\left(\exp(-c(\log x)^{1/3-\varepsilon})\right).
\]
Substituting this back into the above, we have that
\[
\frac{2}{q\Delta}\sum_{x<n\leq x+q\Delta/2}\left(\varpi_2(n)-\sum_{P<p\leq P^{1+\delta}}\frac{1}{p}\right)
=O\left(\exp(-c(\log x)^{1/3-\varepsilon})\right).
\]
Returning to the integral and summing over $q$, we find that the contribution of the principal character to $B_2(X)$ is
\[
\ll X Q_0\exp(-c'(\log X)^{1/3-\varepsilon}),
\]
which is acceptable by the choice of $Q_0$.

We now consider the case $\chi\neq\chi_0$. By the definition of $\varpi_2$ and the Cauchy-Schwarz inequality, we have that
\begin{equation}\label{eq:B2CS}
\begin{aligned}
B_2(X)
&=\sum_{q\leq Q_0}\frac{q}{\varphi(q)}\sum_{\substack{\chi(q)\\\chi\neq\chi_0}}\int_X^{2X}\left|\frac{2}{q\Delta}\sum_{x<n\leq x+q\Delta/2}\chi(n)\varpi_2(n)\right|^2dx\\
&=\sum_{q\leq Q_0}\frac{q}{\varphi(q)}\sum_{\substack{\chi(q)\\\chi\neq\chi_0}}\int_X^{2X}\Bigg|\frac{2}{q\Delta}\sum_{P<p_1\leq P^{1+\delta}}\chi(p_1)\sum_{\frac{x}{p_1}<p_2\leq\frac{x+q\Delta/2}{p_1}}\chi(p_2)\log p_2\Bigg|^2dx\\
&\ll\frac{P^{1+\delta}}{\log P}\sum_{q\leq Q_0}\frac{q}{\varphi(q)}\sum_{P<p_1\leq P^{1+\delta}}\frac{4}{(q\Delta)^2}\sum_{\substack{\chi(q)\\\chi\neq\chi_0}}\int_X^{2X}\Bigg|\sum_{\frac{x}{p_1}<p_2\leq\frac{x+q\Delta/2}{p_1}}\chi(p_2)\log p_2\Bigg|^2dx.
\end{aligned}
\end{equation}
We make the change of variables $u=x/p_1$ to the integral, so that
\[
\int_X^{2X}\Bigg|\sum_{\frac{x}{p_1}<p_2\leq\frac{x+q\Delta/2}{p_1}}\chi(p_2)\log p_2\Bigg|^2dx
=p_1\int_{X/p_1}^{2X/p_1}\Bigg|\sum_{u<p_2\leq u+\frac{q\Delta}{2p_1}}\chi(p_2)\log p_2\Bigg|^2du.
\]
First, in the case $H\leq\exp((\log X)^{\varepsilon^3})$, we now apply Lemma \ref{lem:primesalmostall} to get that
\[
B_2(X)\ll\frac{P^{1+\delta}}{\log P}\frac{X Q_0\log\log Q_0}{\log^DX}\sum_{P<p\leq P^{1+\delta}}\frac{1}{p^2}
\ll\frac{X}{\log^C X}
\]
for $C>0$, as required. In the case $H>\exp((\log X)^{\varepsilon^3})$, we split the sum over $P<p_1\leq P^{1+\delta}$ in \eqref{eq:B2CS} into dyadic intervals and again apply Lemma \ref{lem:primesalmostall} to obtain the required bound.
\end{proof}

\subsection{Bounding \texorpdfstring{$B_1(X)$}{B1(X)}}
It now remains to prove the required bound for $B_1(X)$. This problem can be reduced to finding cancellation in the mean square of a Dirichlet polynomial.

\begin{prop} \label{prop:B1Xbound}
Let $\varepsilon>0$ be fixed sufficiently small. With $B_1(X)$ as defined in \eqref{eq:B1def}, there exists some $\eta=\eta(\varepsilon)>0$ such that
\[
B_1(X)\ll\frac{X}{\log^\eta X}.
\]
\end{prop}

To prove this result, we will need the following variant of a result of Ter\"av\"ainen \cite{teravainen} on the mean square of the Dirichlet polynomial
\begin{equation}\label{eq:Dirichletpolydef}
F(s,\chi):=\sum_{\substack{X<p_1p_2\leq2X\\P<p_1\leq P^{1+\delta}}}\frac{\chi(p_1)\chi(p_2)}{(p_1p_2)^s},
\end{equation}
 to be proved in Section \ref{sec:mvdp}:
 
\begin{prop}\label{prop:Fmvtbound}
Let $\varepsilon>0$ be fixed sufficiently small. Define $T_0=X^{1/100}$ and $F(s,\chi)$ to be the Dirichlet polynomial defined in \eqref{eq:Dirichletpolydef}, with $P$ and $\delta>0$ as in Section \ref{sec:setup}. Then, for $T\geq T_0$, there exists some $\eta=\eta(\varepsilon)>0$ such that
\begin{equation}\label{eq:qFmv}
B_3(X):=\sum_{q\leq Q_0}\frac{q}{\varphi(q)}\sum_{\chi(q)}\int_{T_0}^{T}|F(1+it,\chi)|^2dt
\ll \frac{1}{Q_0\log^{2+\eta}X} \sum_{q\leq Q_0} \left(\frac{qTP\log X}{X}+\frac{q}{\varphi(q)}\right).
\end{equation}
\end{prop}

\begin{proof}[Proof of Proposition \ref{prop:B1Xbound} assuming Proposition \ref{prop:Fmvtbound}.]
First we consider when $\chi=\chi_0$ as we have a different summand in this case. We have 
\begin{equation}\label{eq:principalsum}
\frac{2}{qQ}\sum_{x<n\leq x+qQ/2}\left(\varpi_2(n)-\sum_{P<p\leq P^{1+\delta}}\frac{1}{p}\right)
-\frac{2}{q\Delta}\sum_{x<n\leq x+q\Delta/2}\left(\varpi_2(n)-\sum_{P<p\leq P^{1+\delta}}\frac{1}{p}\right).
\end{equation}
We first consider the contribution of the second and fourth terms, namely
\begin{equation}\label{eq:principalsum2}
\sum_{P<p\leq P^{1+\delta}}\frac{1}{p}\left(\frac{2}{q\Delta}\sum_{x<n\leq x+q\Delta/2}1
-\frac{2}{qQ}\sum_{x<n\leq x+qQ/2}1\right)
\ll\frac{1}{qQ}+\frac{1}{q\Delta}.
\end{equation}
Returning to our expression for $B_1(X)$, by our choice of $Q_0,Q$ and $\Delta$ we have that \eqref{eq:principalsum2} contributes
\[
\ll\sum_{q\leq Q_0}\frac{q}{\varphi(q)}\int_X^{2X}\frac{1}{(qQ)^2}dx
\ll\frac{X}{Q^2}\sum_{q\leq Q_0}\frac{1}{q\varphi(q)}
\ll\frac{X}{Q^2},
\]
which is acceptable. Therefore, when considering the principal character $\chi_0$, we need only to bound
\begin{align*}
&\sum_{q\leq Q_0}\frac{q}{\varphi(q)}\sum_{\chi(q)}\int_X^{2X}\left|\frac{2}{qQ}\sum_{x<n\leq x+qQ/2}\varpi_2(n)
-\frac{2}{q\Delta}\sum_{x<n\leq x+q\Delta/2}\varpi_2(n)\right|^2dx\\
=&\sum_{q\leq Q_0}\frac{q}{\varphi(q)}\sum_{\chi(q)}\int_X^{2X}\left|\frac{2}{qQ}\sum_{x<n\leq x+qQ/2}\varpi_2(n)\chi_0(n)
-\frac{2}{q\Delta}\sum_{x<n\leq x+q\Delta/2}\varpi_2(n)\chi_0(n)\right|^2dx,
\end{align*}
noting that in the range of summation we must have $(n,q)=1$, i.e. $\varpi_2(n)\chi_0(n)=\varpi_2(n)$ for each $X<n\leq2X$. Thus, from now on we are able to unify the treatment of the principal character $\chi_0$ with the rest of the characters modulo $q$ at the cost of a negligible error.

We now apply Lemma \ref{lem:dpbound} with $h_1=qQ/2$ and $h_2=q\Delta/2$ to the integral with respect to $x$ to get
\[
B_1(X)\ll
X\sum_{q\leq Q_0}\frac{q}{\varphi(q)}\sum_{\chi(q)}\Bigg(\frac{\log^2 X}{T_0^2}+\int_{T_0}^{\frac{2X}{qQ}}|F_1(1+it,\chi)|^2dt
+\max_{T\geq\frac{2X}{qQ}}\frac{X}{TqQ}\int_T^{2T}|F_1(1+it,\chi)|^2dt\Bigg),
\]
with $T_0=X^{1/100}$ and
\begin{equation}\label{eq:chardp}
F_1(s,\chi):=\sum_{X<n\leq2X}\frac{\varpi_2(n)\chi(n)}{n^s}=\sum_{\substack{X<p_1p_2\leq2X\\P<p_1\leq P^{1+\delta}}}\frac{\chi(p_1)\chi(p_2)\log p_2}{(p_1p_2)^s}.
\end{equation}
The choice of $T_0$ ensures that the first term is negligible. Applying partial summation, we have that $B_1(X)$ is bounded by
\[
\ll X\log^2 X\sum_{q\leq Q_0}\frac{q}{\varphi(q)}\sum_{\chi(q)}\left(\int_{T_0}^{\frac{2X}{qQ}}|F(1+it,\chi)|^2dt+\max_{T\geq\frac{2X}{qQ}}\frac{X}{TqQ}\int_T^{2T}|F(1+it,\chi)|^2dt\right).
\]
We now apply Proposition \ref{prop:Fmvtbound}. Note that we have $P\log X = Q$, so that the first term in our bound for $B_1(X)$ is bounded by
\[
\ll \frac{X}{Q_0\log^{\eta}X} \sum_{q\leq Q_0} \left(\frac{P\log X}{Q}+\frac{q}{\varphi(q)}\right)\ll\frac{X}{\log^\eta X},
\]
as needed. For the second term, we want to bound
\[
\frac{X^2\log^2X}{Q}\sum_{q\leq Q_0}\frac{1}{\varphi(q)}\max_{T\geq \frac{2X}{qQ}}\frac{1}{T}\sum_{\chi(q)}\int_T^{2T}|F(1+it,\chi)|^2dt.
\]
Applying Proposition \ref{prop:Fmvtbound}, we have the bound 
\[
\ll \frac{X^2}{Q_0Q\log^\eta X}\sum_{q\leq Q_0} \max_{T\geq\frac{2X}{qQ}}\left(\frac{P\log X}{X}+\frac{1}{T\varphi(q)}\right)
\ll\frac{X}{\log^\eta X},
\]
again using that $P\log X=Q$. Overall we have that 
\[
B_1(X)\ll\frac{X}{\log^\eta X},
\]
for some $\eta=\eta(\varepsilon)>0$, as required.
\end{proof}

\section{Preliminaries on Dirichlet Polynomials}
\label{sec:PrelimDP}

Before we can prove Proposition \ref{prop:Fmvtbound}, we first need the following preliminary lemmas on Dirichlet polynomials. 

\subsection{Pointwise Bound}

After we factorise our Dirichlet polynomial, there will be instances where the best we can do is use a pointwise bound. Before we state this bound, we need the following definition of a well-spaced set.

\begin{defn}[Well-Spaced Set]\label{def:wellspaced}
We say a set $\mathcal{T}$ is \textit{well-spaced} if for any $t,u\in\mathcal{T}$ with $t\neq u$ we have that $|t-u|\geq1$.
\end{defn}

\begin{lemma}[Pointwise Bound]\label{lem:pointwise}
Let $\mathcal{S}$ be a set of pairs $(t,\chi)$ with $t\in[-T,T]$ and $\chi$ a Dirichlet character mod $q$ which is well-spaced (i.e. if $(t,\chi),(u,\chi)\in\mathcal{S}$ then $|t-u|\geq1$). Suppose that $\min\{|t|:(t,\chi)\in\mathcal{S}\}\gg\log^AN$ for all $A>0$ if $\chi=\chi_0$. Let
\[
P(s,\chi):=\sum_{\substack{N<p_1\cdots p_k\leq2N\\p_1,\ldots,p_k\geq z}}\frac{\chi(p_1)\cdots \chi(p_k)}{(p_1\cdots p_k)^{1+it}},
\]
where $z\geq\exp(\log^{9/10}N)$. Then for any $C>0$ we have
\[
|P(1+it,\chi)|\ll\frac{1}{\log^CN}.
\]
\end{lemma}

\begin{proof}
This is \cite[Lemma~10.7]{harmansieves}.
\end{proof}

\begin{defn}[Prime-factored polynomial, \cite{teravainen}]\label{def:primefactored}
Let $M\geq1$ and
\[
M(s,\chi)=\sum_{M<m\leq2M}\frac{a_m\chi(m)}{m^s}
\]
be a Dirichlet polynomial with $|a_m|\ll d_r(m)$ for some fixed $r$. We say that $M(s,\chi)$ is \textit{prime-factored} if for each $C>0$ we have 
\[
\sup_{(t,\chi)\in\mathcal{S}}|M(1+it,\chi)|\ll\frac{1}{\log^CM}
\]
when $\exp((\log M)^{1/3})\leq t\leq M^{C\log\log M}$, where $\mathcal{S}$ is as defined in the previous lemma.
\end{defn}

\subsection{Decomposing Dirichlet Polynomials}
As in the work of Ter{\"a}v{\"a}inen \cite{teravainen} and Matom{\"a}ki, Radziwi{\l}{\l} \cite{mrmult}, we take advantage of the bilinear structure to factorise our Dirichlet polynomial.

\begin{lemma}[Factorisation of Dirichlet Polynomials]\label{lem:factoriseF}
Define 
\[
F(s):=\sum_{\substack{X<mn\leq2X\\M\leq m\leq M'}}\frac{a_mb_n}{(mn)^s}
\]
for some $M'>M\geq2$ and arbitrary complex numbers $a_m,b_n$. Let $U\geq1$ and define
\[
A_v(s):=\sum_{e^{\frac{v}{U}}\leq m<e^{\frac{v+1}{U}}}\frac{a_m}{m^s},\quad
B_v(s):=\sum_{Xe^{-\frac{v}{U}}<n\leq2Xe^{-\frac{v}{U}}}\frac{b_n}{n^s}.
\]
Then
\begin{equation}\label{eq:factoredF}
F(s)=\sum_{v\in I\cap\mathbb{Z}}A_v(s)B_v(s)+\sum_{\substack{k\in[Xe^{-1/U},Xe^{1/U}]\\\text{or }k\in[2X,2Xe^{1/U}]}}\frac{d_k}{k^s}
\end{equation}
where $I=[U\log M,U\log M']$ and 
\begin{equation}\label{eq:dkdefn}
|d_k|\leq\sum_{k=mn}|a_mb_n|.
\end{equation}
\end{lemma}

\begin{proof}
This is \cite[Lemma~2]{teravainen} (see also \cite[Lemma~12]{mrmult}).
\end{proof}

In some cases we will use the Heath-Brown identity to decompose a long polynomial into products of shorter polynomials. 

\begin{lemma}[Heath-Brown decomposition]\label{lem:heathbrown}
Let $k\geq1$ be a fixed integer, $T\geq2$ and fix $\varepsilon>0$. Define the Dirichlet polynomial $P(s,\chi):=\sum_{P\leq p<P'}\chi(p)p^{-s}$ with $P\gg T^\varepsilon$, $P'\in\left[P+\frac{P}{\log T},2P\right]$. Then, there exist Dirichlet polynomials $Q_1(s,\chi),\ldots,Q_L(s,\chi)$ and a constant $C>0$ such that $L\leq\log^{C}X$ and 
\[
|P(1+it,\chi)|\ll(\log^{C}X)(|Q_1(1+it,\chi)|+\cdots+|Q_L(1+it,\chi)|)
\]
for all $t\in[-T,T]$. Here, each $Q_j(s,\chi)$ is of the form
\[
Q_j(s,\chi)=\prod_{i\leq J_j}M_i(s,\chi),\quad J_j\leq 2k,
\]
where each $M_i(s,\chi)$ is a prime-factored Dirichlet polynomial (depending on $j$) of the form
\[
\sum_{M_i<n\leq2M_i}\frac{\chi(n)\log n}{n^s},\;\sum_{M_i<n\leq2M_i}\frac{\chi(n)}{n^s},\text{ or }\sum_{M_i<n\leq2M_i}\frac{\mu(n)\chi(n)}{n^s},
\]
whose lengths satisfy $M_1\cdots M_J=X^{1+o(1)},M_i\gg\exp\left(\frac{\log P}{\log\log P}\right)$. Furthermore, if in fact $M_i>X^{1/k}$, then $M_i(s,\chi)$ is of the form
\begin{equation}\label{eq:dirichletlsum}
\sum_{M_i<n\leq2M_i}\frac{\chi(n)\log n}{n^s}\text{ or }\sum_{M_i<n\leq2M_i}\frac{\chi(n)}{n^s}.
\end{equation}
\end{lemma}

\begin{proof}
This is the Dirichlet character analogue of \cite[Lemma 10]{teravainen}, which follows from the same argument.
\end{proof}

\subsection{Mean Value Theorems for Dirichlet Polynomials}
Now we state two mean value theorems, the first being the classical result:

\begin{lemma}[Mean Value Theorem]\label{lem:mvt}
Let $q,X\geq 1$ and let $a_n$ be arbitrary complex numbers with $F(s,\chi):=\sum_{X<n\leq2X}\frac{a_n\chi(n)}{n^s}$. Then
\[
\sum_{\chi(q)}\int_{-T}^{T}|F(it,\chi)|^2dt
\ll\left(\varphi(q)T+\frac{\varphi(q)}{q}X\right)\sum_{\substack{X<n\leq2X\\(n,q)=1}}|a_n|^2.
\]
\end{lemma}

\begin{proof}
See, for example, \cite[Chapter 6, Eq. (6.14)]{montgomery}.
\end{proof}

Next we state a variant of the mean value theorem which will allow us to save a $\log X$ in certain parts of the proof.

\begin{lemma}\label{lem:mvtvariant}
With the same assumptions as Lemma \ref{lem:mvt}, we have that
\[
\sum_{\chi(q)}\int_{-T}^{T}|F(it,\chi)|^2dt
\ll T\varphi(q)\Bigg(\sum_{\substack{X<n\leq2X\\(n,q)=1}} |a_n|^2+\sum_{\substack{1\leq h\leq\frac{X}{T}\\q\mid h}}\sum_{\substack{X<n\leq2X\\(n(n+h),q)=1}}|a_{n+h}||a_n|\Bigg).
\]
\end{lemma}

\begin{proof}
This is the Dirichlet character analogue of \cite[Lemma~4]{teravainen}, which follows from \cite[Lemma~7.1]{iwanieckowalski}. The proof is contained in the proof of \cite[Lemma 5.2]{mrt2}.
\end{proof}

After factorising the Dirichlet polynomial $F$ and splitting the domain of integration according to the size of the factors, there will be cases where the mean value is taken over a well-spaced set. In this case, we will apply the Hal{\'a}sz-Montgomery inequality:

\begin{lemma}[Hal{\'a}sz-Montgomery Inequality]\label{lem:HalaszMont}
Let $T\geq1$, $q\geq2$. Let $\mathcal{S}$ be a well-spaced set of pairs $(t,\chi)$ with $t\in[-T,T]$ and $\chi$ a Dirichlet character mod $q$. With the same assumptions as Lemma \ref{lem:mvt}, we have that
\[
\sum_{(t,\chi)\in\mathcal{S}}|F(it,\chi)|^2
\ll\left(\frac{\varphi(q)}{q}X+|\mathcal{S}|(qT)^{1/2}\right)(\log(2qT))\sum_{\substack{X<n\leq2X\\(n,q)=1}} |a_n|^2.
\]
\end{lemma}

\begin{proof}
This is \cite[Lemma 7.4]{klurman}.
\end{proof}

\subsection{Large Value Theorems}
There will be subsets of the domain of integration where a short Dirichlet polynomial factor is large, in which case we apply the following large value theorem.

\begin{lemma}[Large Value Theorem]\label{lem:largevalues}
Let $P\geq1,V>0$, $|a_p|\leq1$ and $F(s,\chi)=\sum_{P<p\leq2P}\frac{a_p\chi(p)}{p^s}$. Let $\mathcal{S}\subset[-T,T]\times\{\chi\mod q\}$ be a well-spaced set such that $|F(1+it,\chi)|\geq V$ for all $(t,\chi)\in\mathcal{S}$. Then
\[
|\mathcal{S}|\ll (qT)^{\frac{2\log(1/V)}{\log P}}V^{-2}\exp\left((1+o(1))\frac{\log(qT)\log\log(qT)}{\log P}\right).
\]
\end{lemma}

\begin{proof}
This is the Dirichlet character analogue of \cite[Lemma~6]{teravainen} and \cite[Lemma~8]{mrmult}. Also see \cite[Lemma 7.5]{klurman}.
\end{proof}

\begin{remark}
As remarked in \cite[Remark 6]{teravainen}, this lemma can still be applied to polynomials with coefficients not only supported on the primes as long as we have $P\gg X^{\varepsilon}$, as will be the case in our application.
\end{remark}

Alternatively, in the case that we have a longer Dirichlet polynomial factor which is large, we will apply a result of Jutila on large values.

\begin{lemma}[Jutila's Large Value Theorem]\label{lem:Jutilalarge}
Let $\varepsilon>0$ be fixed, $|a_n|\leq d_r(n)$ for some fixed $r$ and $F(s,\chi)=\sum_{X<n\leq2X}\frac{a_n\chi(n)}{n^s}$. Let $k$ be a fixed positive integer and $\mathcal{S}\subset[-T,T]\times\{\chi\mod q\}$ be a well-spaced set such that $|F(1+it,\chi)|\geq V$ for all $(t,\chi)\in\mathcal{S}$. Then,
\[
|\mathcal{S}|\ll\left(V^{-2}+\left(\frac{qTV^{-4}}{X^2}\right)^k+\frac{qTV^{-8k}}{X^{2k}}\right)(qTX)^{\varepsilon}.
\]
\end{lemma}

\begin{proof}
This is the first bound of the main theorem in \cite{jutila}.
\end{proof}

\subsection{Moments of Dirichlet Polynomials}
After decomposing the Dirichlet polynomial using the Heath-Brown decomposition (Lemma~\ref{lem:heathbrown}), we can have a long polynomial which is the partial sum of a Dirichlet $L$-function (or its derivative). In this case, we will apply the Cauchy-Schwarz inequality to enable us to use the following bound on the twisted fourth moment of such sums:

\begin{lemma}[Twisted Fourth Moment Estimate] \label{lem:fourthmoment}
Let $Q_0\leq T^\varepsilon$, $T^\varepsilon\leq T_0\leq T$, $1\leq M,N\leq T^{1+o(1)}$ and define the Dirichlet polynomials
\begin{align*}
N(s,\chi)&=\sum_{N<n\leq2N}\frac{\chi(n)}{n^s}\text{ or }\sum_{N<n\leq2N}\frac{\chi(n)\log n}{n^s},\\
M(s,\chi)&=\sum_{M<m\leq2M}\frac{a_m\chi(m)}{m^s},
\end{align*}
with $a_m$ any complex numbers. Then we have that
\begin{align}
\sum_{q\leq Q_0}\frac{1}{\varphi(q)}\sum_{\chi(q)}\int_{T_0}^{T} & |N(1+it,\chi)|^4|M(1+it,\chi)|^2dt\nonumber\\
&\ll\left(\frac{Q_0T}{MN^2}(1+M^2(Q_0T)^{-1/2})+\frac{1}{T_0}\right)(Q_0T)^{\varepsilon}\max_{M<m\leq2M}|a_m|^2.
\label{eq:fourthmomentresult}
\end{align}
\end{lemma}

\begin{proof}
This is the Dirichlet character analogue of \cite[Lemma 9]{teravainen} and we follow the same argument.
In the case $1\leq qt\leq N$, we use partial summation and the hybrid result of Fujii, Gallagher and Montgomery \cite{hybridbound}
\[
\sum_{n\leq N}\chi(n)n^{it}=\frac{\delta_\chi \varphi(q)N^{1+it}}{q(1+it)}+O((q\tau)^{1/2}\log(q\tau)),
\]
with $\tau:=|t|+2$ in place of the zeta sum bound to get that
\begin{align*}
\sum_{q\leq Q_0}&\frac{1}{\varphi(q)}\sum_{\chi(q)}\int_{T_0}^{T}|N(1+it,\chi)|^4|M(1+it,\chi)|^2dt\\
&\ll T^{\varepsilon}\max_{M<m\leq2M}|a_m|^2\sum_{q\leq Q_0}\frac{1}{\varphi(q)}\sum_{\chi(q)}\int_{T_0}^T\left(\frac{\varphi(q)}{q(1+|t|)}\right)^4+\frac{\log^4(q\tau)}{(q\tau)^2}dt\\
&\ll \frac{T^{\varepsilon}}{T_0}\max_{M<m\leq2M}|a_m|^2,
\end{align*}
providing the third term of \eqref{eq:fourthmomentresult}.

In the case $N\leq qt\leq Q_0T$, we apply in place of Watt's twisted moment result its Dirichlet character analogue \cite[Theorem 2]{harmanwattwong} to obtain the first two terms of \eqref{eq:fourthmomentresult}.
\end{proof}

\section{Bounding the Mean Value of a Dirichlet Polynomial}
\label{sec:mvdp}

We are now able to prove Proposition \ref{prop:Fmvtbound}, completing the proof of Proposition \ref{prop:B2bound} and consequently Theorem \ref{thm:weightedmainthm}. We will adapt the argument appearing in \cite[Sections 2-4]{teravainen}. We first factorise the Dirichlet polynomial $F(s,\chi)$ before bounding the contribution of the remainder terms, that is, the second term of \eqref{eq:factoredF}.

\begin{lemma}\label{lem:factorisedFLemma}
Let $\varepsilon>0$ be fixed sufficiently small and $T_0=X^{1/100}$. Denote
\[
G_v(s,\chi):=\sum_{e^{\frac{v}{U}}<p\leq e^{\frac{v+1}{U}}}\frac{\chi(p)}{p^s},\quad
H_v(s,\chi):=\sum_{Xe^{-\frac{v}{U}}<p\leq2Xe^{-\frac{v}{U}}}\frac{\chi(p)}{p^s},
\]
then we have the bound
\begin{align*}
B_3(X)\ll \sum_{q\leq Q_0}\Bigg(\frac{qU^2\log^2P}{\varphi(q)}\sum_{\chi(q)}\int_{T_0}^T & |G_{v_0}(1+it,\chi)|^2|H_{v_0}(1+it,\chi)|^2dt\\
&+\frac{1}{Q_0\log^{2+\eta}X}\left(\frac{qT\log X}{X}+\frac{q}{\varphi(q)}\right)\Bigg),
\end{align*}
for some $\eta=\eta(\varepsilon)>0$, where we take $I=[U\log P,(1+\delta)U\log P]$, $U:=Q_0^{1+\varepsilon^2}$ and $v_0\in I$ a suitable integer.
\end{lemma}

\begin{proof}
We factorise $F(s,\chi)$ using Lemma \ref{lem:factoriseF} to get that
\[
F(s,\chi)=\sum_{v\in I\cap\mathbb{Z}}G_v(s,\chi)H_v(s,\chi) +\sum_{\substack{k\in[Xe^{-1/U},Xe^{1/U}]\\\text{or }k\in[2X,2Xe^{1/U}]}}\frac{d_k\chi(k)}{k^s},
\]
where $I=[U\log P,(1+\delta)U\log P]$, $U:=Q_0^{1+\varepsilon^2}$ and
\[
|d_k|\leq\sum_{\substack{k=p_1p_2\\P<p_1\leq P^{1+\delta}}}1.
\]

Therefore, taking the maximum in the sum over $I$, the mean square of the Dirichlet polynomial is bounded by
\begin{align*}
\int_{T_0}^{T}|F(1+it,\chi)|^2dt
&\ll\int_{T_0}^T\left|\sum_{v\in I\cap\mathbb{Z}}G_v(1+it,\chi)H_v(1+it,\chi)\right|^2dt+\int_{T_0}^{T}\Bigg|\sum_{\substack{k\in[Xe^{-1/U},Xe^{1/U}]\\\text{or }k\in[2X,2Xe^{1/U}]}}\frac{d_k\chi(k)}{k^{1+it}}\Bigg|^2dt\\
&\ll|I|^2\int_{T_0}^T|G_{v_0}(1+it,\chi)|^2|H_{v_0}(1+it,\chi)|^2dt+\int_{T_0}^{T}\Bigg|\sum_{\substack{k\in[Xe^{-1/U},Xe^{1/U}]\\\text{or }k\in[2X,2Xe^{1/U}]}}\frac{d_k\chi(k)}{k^{1+it}}\Bigg|^2dt,
\end{align*}
where $v_0\in I$ is the integer maximising the right hand side. Applying Lemma \ref{lem:mvtvariant} to the second integral, we have that
\begin{equation}
\begin{aligned}
\sum_{\chi(q)}\int_{T_0}^{T}|F(1+it,\chi)|^2dt
\ll& U^2\log^2P\sum_{\chi(q)}\int_{T_0}^T|G_{v_0}(1+it,\chi)|^2|H_{v_0}(1+it,\chi)|^2dt\\
&+T\varphi(q)\sum_{\substack{k\in[Xe^{-1/U},Xe^{1/U}]\\\text{or }k\in[2X,2Xe^{1/U}]\\(k,q)=1}}\frac{|d_k|^2}{k^2}
+T\varphi(q)\sum_{\substack{1\leq h\leq\frac{2Xe^{1/U}}{T}\\q\mid h}}\sum_{\substack{m-n=h\\m,n\in[Xe^{-1/U},Xe^{1/U}]\\\text{or }m,n\in[2X,2Xe^{1/U}]\\(mn,q)=1}}\frac{|d_m||d_n|}{mn}.
\label{eq:factorisedintegral}
\end{aligned}
\end{equation}
We now bound the last two terms. We consider only the sums where $k\in[Xe^{-1/U},Xe^{1/U}]$, with the sums over $k\in[2X,2Xe^{1/U}]$ being treated analogously. For the first sum, 
\begin{equation}\label{eq:brun1}
\sum_{\substack{k=p_1p_2\\Xe^{-1/U}\leq k\leq Xe^{1/U}\\P<p_1\leq P^{1+\delta}}}\frac{1}{k^2}
\ll\frac{e^{2/U}}{X^2}\sum_{P<p_1\leq P^{1+\delta}}\sum_{\frac{Xe^{-1/U}}{p_1}\leq p_2\leq\frac{Xe^{1/U}}{p_1}}1.
\end{equation}
By the Brun-Titchmarsh inequality, we have the bound
\[
\sum_{\frac{Xe^{-1/U}}{p_1}\leq p_2\leq\frac{Xe^{1/U}}{p_1}}1
\ll\frac{X(e^{1/U}-e^{-1/U})}{p_1\log X}.
\]
Returning to \eqref{eq:brun1}, by Mertens' theorem we have that
\[
\sum_{\substack{Xe^{-1/U}\leq p_1p_2\leq Xe^{1/U}\\P<p_1\leq P^{1+\delta}}}\frac{1}{(p_1p_2)^2}
\ll\frac{e^{3/U}-e^{1/U}}{X\log X}\sum_{P<p\leq P^{1+\delta}}\frac{1}{p}
\ll\frac{1}{XU\log X}.
\]

We will use Brun's sieve to bound the second of these sums. We may trivially bound
\[
\begin{aligned}
&|\{n\leq2X:n=p_1p_2, p_1\in(P,P^{1+\delta}]\}|\\
&\ll|\{n\leq2X:n=p_1m, p_1\in(P,P^{1+\delta}], (m,P(z))=1\}|
\end{aligned}
\]
where we define $P(z)=\prod_{p<z}p$ with $z=X^{1/\beta}$ and $\beta>1$ suitably large. Let $\Pi$ be the product of all primes in $\tilde{I}:=(P,P^{1+\delta}]\cap[1,z)$ and $P'(z)=\prod_{p<z, p\nmid h}p$. Therefore, we have that
\begin{align*}
\sum_{\substack{1\leq h\leq\frac{2Xe^{1/U}}{T} \\ q\mid h}}&\sum_{Xe^{-1/U}\leq m\leq Xe^{1/U}}\frac{|d_m||d_{m+h}|}{m(m+h)}\\
&\ll\frac{e^{2/U}}{X^2}\sum_{\substack{1\leq h\leq\frac{2Xe^{1/U}}{T} \\ q\mid h}}\left|\left\{m\in[Xe^{-1/U},Xe^{1/U}]:\left(m(m+h),\frac{P'(z)}{\Pi}\right)=1\right\}\right|.
\end{align*}
Brun's sieve then gives the bound
\begin{align*}
\left|\left\{m \in[Xe^{-1/U},Xe^{1/U}]:\left(m(m+h),\frac{P'(z)}{\Pi}\right)=1\right\}\right|
&\ll \frac{hX(e^{1/U}-e^{-1/U})}{\varphi(h)}\prod_{\substack{p<z\\p\notin\tilde{I}}}\left(1-\frac{2}{p}\right)\\
&\ll\frac{X}{U\log^2z}\frac{h}{\varphi(h)}.
\end{align*}
Therefore we have
\[
\begin{aligned}
\sum_{\substack{1\leq h\leq\frac{2Xe^{1/U}}{T} \\ q\mid h}}\sum_{Xe^{-1/U}\leq m\leq Xe^{1/U}}\frac{|d_m||d_{m+h}|}{m(m+h)}
&\ll\frac{e^{2/U}}{XU\log^2z}\sum_{\substack{1\leq h\leq\frac{2Xe^{1/U}}{T} \\ q\mid h}}\frac{h}{\varphi(h)}\\
&\ll\frac{1}{\varphi(q)TU\log^2X}.
\end{aligned}
\]
Combining these estimates and applying the definition of $U$, these two sums contribute 
\[
\ll\left(\frac{\varphi(q)T\log X}{X}+1\right)\frac{1}{Q_0\log^{2+\eta}X}
\]
to \eqref{eq:factorisedintegral} for some $\eta=\eta(\varepsilon)>0$, as needed.
\end{proof}

It remains to estimate the integral appearing in Lemma \ref{lem:factorisedFLemma}. We split the domain of integration $[T_0,T]$ according to the size of the polynomial $G_{v_0}$. We will first bound the contribution of $\mathcal{S}_1\subset[T_0,T]\times\{\chi\mod q\}$ defined by
\begin{equation}\label{eq:S1defn}
\mathcal{S}_1:=\{(t,\chi)\in[T_0,T]\times\{\chi\mod q\}:|G_{v_0}(1+it,\chi)|\leq e^{-\frac{\alpha_1v_0}{U}}\},
\end{equation}
where $\alpha_1:=\frac{3}{34}-\varepsilon'$ and $\varepsilon'>0$ is sufficiently small in terms of $\varepsilon>0$. We may write
\begin{equation}\label{eq:T1defn}
\mathcal{S}_1=\bigcup_{\chi\mod q}\{\chi\}\times\mathcal{T}_{1,\chi}
\end{equation}
for some $\mathcal{T}_{1,\chi}\subset[T_0,T]$.

\subsection{The contribution of \texorpdfstring{$\mathcal{S}_1$}{S1}}

We first treat the contribution of the integral over $\mathcal{T}_{1,\chi}$, where the polynomial $G_{v_0}(1+it,\chi)$ is pointwise small.

\begin{lemma}\label{lem:T1integral}
Let $\varepsilon>0$ be fixed sufficiently small and $\mathcal{T}_{1,\chi}$ be defined as in \eqref{eq:T1defn}. Then, there exists some $\eta=\eta(\varepsilon)>0$ such that
\[
U^2\log^2P\sum_{q\leq Q_0}\frac{q}{\varphi(q)}\sum_{\chi(q)}\int_{\mathcal{T}_{1,\chi}}|G_{v_0}(1+it)|^2|H_{v_0}(1+it)|^2dt
\ll\frac{1}{Q_0\log^{2+\eta}X}\sum_{q\leq Q_0} \left(\frac{qPT\log X}{X}+1\right).
\]
\end{lemma}

\begin{proof}
First we apply the definition of $\mathcal{T}_{1,\chi}$, bounding pointwise $|G_{v_0}(1+it,\chi)|\leq e^{-\frac{\alpha_1v_0}{U}}\leq P^{-\alpha_1}$ to bound the integral over $\mathcal{T}_{1,\chi}$ by
\[
\int_{\mathcal{T}_{1,\chi}}|G_{v_0}(1+it,\chi)|^2|H_{v_0}(1+it,\chi)|^2dt
\ll P^{-2\alpha_1}\int_{\mathcal{T}_{1,\chi}}|H_{v_0}(1+it,\chi)|^2dt.
\]
Applying Lemma \ref{lem:mvtvariant}, we have that
\begin{align*}
\sum_{\chi(q)}&\int_{\mathcal{T}_{1,\chi}}|G_{v_0}(1+it,\chi)|^2|H_{v_0}(1+it,\chi)|^2dt\\
&\ll \frac{P^{-2\alpha_1}T\varphi(q)e^{2v_0/U}}{X^2}\Bigg(\sum_{\frac{X}{e^{v_0/U}}<p\leq\frac{2X}{e^{v_0/U}}}1+\sum_{\substack{1\leq h\leq\frac{X}{Te^{v_0/U}}\\q\mid h}}\sum_{\substack{\frac{X}{e^{v_0/U}}<p_1,p_2\leq\frac{2X}{e^{v_0/U}}\\p_1-p_2=h}}1\Bigg).
\end{align*}
By Chebyshev's inequality, we have that
\[
\sum_{\frac{X}{e^{v_0/U}}<p\leq\frac{2X}{e^{v_0/U}}}1
\ll\frac{X}{e^{v_0/U}\log X}.
\]
For the second term, we have by Lemma \ref{lem:primepairs} that
\[
\sum_{\substack{1\leq h\leq\frac{X}{Te^{v_0/U}}\\q\mid h}}\sum_{\substack{\frac{X}{e^{v_0/U}}<p_1,p_2\leq\frac{2X}{e^{v_0/U}}\\p_1-p_2=h}}1
\ll\frac{X}{e^{v_0/U}\log^2X}\sum_{\substack{1\leq h\leq\frac{X}{Te^{v_0/U}} \\ q\mid h}}\mathfrak{S}(h)
\ll\frac{X^2}{e^{2v_0/U}qT\log^2X}.
\]
Combining these estimates, we have that
\[
\sum_{\chi(q)}\int_{\mathcal{T}_{1,\chi}}|G_{v_0}(1+it,\chi)|^2|H_{v_0}(1+it,\chi)|^2dt
\ll \frac{\varphi(q)P^{\delta-2\alpha_1}}{q\log^2X}\left(\frac{qPT\log X}{X}+1\right).
\]
Thus the overall contribution to the sum $B_3(X)$ is
\[
\frac{U^2P^{\delta-2\alpha_1}\log^2P}{\log^2X}\sum_{q\leq Q_0}\left(\frac{qPT\log X}{X}+1\right).
\]
Now, by our choices of $P$, $U$ and the definition of $\alpha_1$ with $\varepsilon'$ sufficiently small in terms of $\varepsilon$, we have that
\[
\frac{U^2P^{\delta-2\alpha_1}\log^2P}{\log^2X}
\ll\frac{1}{Q_0\log^{2+\eta}X},
\]
for some $\eta=\eta(\varepsilon)>0$ as $\delta>0$ is sufficiently small.
\end{proof}

\subsection{The contribution of the complement of \texorpdfstring{$\mathcal{S}_1$}{S1}}

It remains to consider the contribution of the complement of $\mathcal{S}_1$. We apply Lemma \ref{lem:heathbrown} to $H_{v_0}(1+it,\chi)$ with $k=3$, decomposing this polynomial into 
\[
|H_{v_0}(1+it,\chi)|
\ll(\log^C X)\left(|Q_1(1+it,\chi)|+\cdots+|Q_L(1+it,\chi)|\right),
\]
where $L\leq \log^CX$ for some $C>0$. Each $Q_j(s,\chi)$ is of the form $Q_j(s,\chi)=\prod_{i\leq J_j}M_i(s,\chi)$ with $J_j\leq6$ for each $1\leq j\leq L$, where $M_i(s,\chi)$ are prime-factored Dirichlet polynomials of the form
\[
\sum_{M_i<n\leq2M_i}\frac{\chi(n)\log n}{n^s},\;\sum_{M_i<n\leq2M_i}\frac{\chi(n)}{n^s},\text{ or }\sum_{M_i<n\leq2M_i}\frac{\mu(n)\chi(n)}{n^s},
\]
whose lengths satisfy $M_1\cdots M_{J_j}=X^{1+o(1)}$, $M_i\gg\exp\left(\frac{\log X}{\log\log X}\right)$ for each $i$. We will treat the polynomials $Q_j(s,\chi)$ according to the lengths $M_i$ of the factors as follows:

\textbf{Type II Sums:}
Suppose we have $Q_j(s,\chi)=\prod_{i\leq J_j}M_i(s,\chi)$ for some $1\leq j\leq L$ with $M_i\leq X^{1/3+\varepsilon'}$ for some $i\leq J_j\leq6$. Then, we rewrite $Q_j(s,\chi)=M_1(s,\chi)M_2(s,\chi)$ with $\exp\left(\frac{\log X}{\log\log X}\right)\ll M_1\leq X^{1/3+\varepsilon'}$ and $M_2=X^{1+o(1)}/M_1$. Note we may write $M_1=X^{\nu}$ for some $0<\nu\leq1/3+\varepsilon'$. Where the coefficient $\log n$ appears, we apply partial summation. The polynomial $M_2(s,\chi)$ is a product of polynomials, and the coefficients are given by convolving coefficients which are one of the sequences $(\mu(n)), (1)$. Thus the coefficients of the polynomial $M_2(s,\chi)$ are bounded in absolute value by $\ll d_r(n)$ with $r\leq5$.

\textbf{Type I Sums:} 
Otherwise, we may write $Q_j(s,\chi)=N_1(s,\chi)N_2(s,\chi)$, where each $N_i(s,\chi)$ is of the form 
\[
\sum_{N_i<n\leq2N_i}\frac{\chi(n)\log n}{n^s},\text{ or }\sum_{N_i<n\leq2N_i}\frac{\chi(n)}{n^s},
\]
with lengths satisfying $N_1N_2=X^{1+o(1)}$. Note that if in fact only one of the lengths $N_i$ satisfies $N_i>X^{1/3+\varepsilon'}$, then one of $N_1(s,\chi), N_2(s,\chi)$ can be the constant polynomial $1^{-s}$. Since we have that $N_1N_2=X^{1+o(1)}$, without loss of generality we may take that $N_1>X^{1/2-\varepsilon'}$, so that $X^{1/3+\varepsilon'}< N_2\leq X^{1/2+\varepsilon'}$.

\subsubsection{Type II Sums}

To treat the contribution of these sums, we split the complement of $\mathcal{S}_1$ according to the size of $M_1(1+it,\chi)$:
\begin{equation}\label{eq:SS2defn}
\begin{aligned}
\mathcal{S}_2&:=\{(t,\chi)\in[T_0,T]\times\{\chi\mod q\}:|M_1(1+it,\chi)|\leq M_1^{-\alpha_2}\}\setminus \mathcal{S}_1,\\
\mathcal{S}&:=([T_0,T]\times\{\chi\mod q\})\setminus (\mathcal{S}_1\cup\mathcal{S}_2),
\end{aligned}
\end{equation}
with $\alpha_2:=\frac{2}{17}-\varepsilon'>\alpha_1$. As before, we may write
\begin{equation}\label{eq:TT2defn}
\begin{aligned}
\mathcal{S}_2&=\bigcup_{\chi\mod q}\{\chi\}\times\mathcal{T}_{2,\chi},\\
\mathcal{S}&=\bigcup_{\chi\mod q}\{\chi\}\times\mathcal{T}_\chi,
\end{aligned}
\end{equation}
for some $\mathcal{T}_{2,\chi},\mathcal{T}_\chi\subset[T_0,T]$. We first consider the contribution of the integral over $\mathcal{T}_{2,\chi}$.

\begin{lemma}\label{lem:T2integral}
Let $\varepsilon>0$ be fixed sufficiently small. Let $\mathcal{T}_{2,\chi}$ be defined as in \eqref{eq:TT2defn} and $M_1(s,\chi),M_2(s,\chi)$ be the prime-factored polynomials defined previously. Then
\[
\sum_{q\leq Q_0}\frac{q}{\varphi(q)}\sum_{\chi(q)}\int_{\mathcal{T}_{2,\chi}}|G_{v_0}(1+it,\chi)|^2|M_1(1+it,\chi)M_2(1+it,\chi)|^2dt
\ll X^{-\varepsilon/1000}.
\]
\end{lemma}

\begin{proof}
By definition of $\mathcal{T}_{2,\chi}$, we have that $(|G_{v_0}(1+it,\chi)|P^{\alpha_1})^{2(k-1)}\geq1$, where $k=\lceil\log M_1/\log P\rceil$. Therefore, we have 
\begin{align}
\sum_{\chi(q)} & \int_{\mathcal{T}_{2,\chi}}|G_{v_0}(1+it,\chi)|^2|M_1(1+it,\chi)M_2(1+it,\chi)|^2dt\nonumber\\
&\ll  M_1^{-2\alpha_2}P^{2\alpha_1(k-1)}\sum_{\chi(q)}\int_{\mathcal{T}_{2,\chi}}|G_{v_0}^{k}(1+it,\chi)M_2(1+it,\chi)|^2dt.
\label{eq:T2contribution}
\end{align}
By the choice of $k$, we have that
\[
P^{2\alpha_1k}
\ll \exp\left(\frac{2\alpha_1\log P \log M_1}{\log P}\right)\\
= M_1^{2\alpha_1}.
\]
Therefore \eqref{eq:T2contribution} is bounded by
\[
\begin{aligned}
&\ll M_1^{2\alpha_1-2\alpha_2}P^{-2\alpha_1}\sum_{\chi(q)}\int_{\mathcal{T}_{2,\chi}}|G_{v_0}^{k}(1+it,\chi)M_2(1+it,\chi)|^2dt\\
&\ll M_1^{2\alpha_1-2\alpha_2}P^{-2\alpha_1}\sum_{\chi(q)}\int_{\mathcal{T}_{2,\chi}}|A(1+it,\chi)|^2dt,
\end{aligned}
\]
where we define 
\[
A(s,\chi):=\sum_{n\in J}\frac{A_n\chi(n)}{n^s},
\]
with $J:=(M_2e^{kv_0/U},2M_2e^{k(v_0+1)/U}]$ and the coefficients $A_n$ satisfying
\[
|A_n|\leq\sum_{\substack{n=p_1\cdots p_km\\e^{v_0/U}<p_i\leq e^{(v_0+1)/U}\\i=1,\ldots,k\\M_2<m\leq2M_2}}d_r(m),
\]
where $r\leq5$, as before. Note that the primes $p_1,\ldots,p_k$ are not necessarily distinct and $m$ may also have prime factors in the range $(e^{v_0/U},e^{(v_0+1)/U}]$. Applying Lemma \ref{lem:mvt} to the integral, we have that
\[
\sum_{\chi(q)}\int_{\mathcal{T}_{2,\chi}}|A(1+it,\chi)|^2dt
\ll\left(\varphi(q)T+\frac{\varphi(q)}{q}M_2e^{kv_0/U}(2e^{k/U}-1)\right)\sum_{\substack{n\in J\\(n,q)=1}}\frac{|A_n|^2}{n^2}.
\]
Following \cite[Lemma 13]{mrmult}, for the coefficients $A_n$ of $A(s,\chi)$, we have the bound
\[
|A_n|
\ll M_2^{o(1)}\sum_{\substack{n=p_1\cdots p_km\\e^{v_0/U}<p_i\leq e^{(v_0+1)/U}\\i=1,\ldots,k\\M_2<m\leq2M_2}}1.
\]
The number of ways we can write $d=p_1\cdots p_k$ with $p_i$ not necessarily distinct is at most $k!$. Then we have the bound 
\[
|A_n|
\ll M_2^{o(1)}k!\sum_{\substack{n=md\\p\mid d\Rightarrow e^{v_0/U}<p\leq e^{(v_0+1)/U}}}1,
\]
trivially extending the range of summation for $m$. We write $g(n)$ for the latter sum, which is multiplicative and satisfies
\[
g(p^a)=
\begin{cases}
a+1,&\text{ if }e^{v_0/U}<p\leq e^{(v_0+1)/U},\\
0,&\text{ otherwise.}
\end{cases}
\]
Therefore, we clearly have $g(n)\ll d(n)\ll n^{o(1)}$ and thus
\begin{align*}
\sum_{\substack{n\in J\\(n,q)=1}}\frac{|A_n|^2}{n^2}
&\ll\frac{X^{o(1)}k!}{e^{kv_0/U}M_2}\sum_{\substack{n\in J\\(n,q)=1}}\frac{|A_n|}{n}\\
&\ll\frac{X^{o(1)}k!}{e^{kv_0/U}M_2}\sum_{M_2<m\leq2M_2}\frac{d_r(m)}{m}\left(\sum_{e^{v_0/U}<p_i\leq e^{(v_0+1)/U}}\frac{1}{p}\right)^{k}\\
&\ll\frac{X^{o(1)}k!}{e^{kv_0/U}M_2}\frac{k}{v_0},
\end{align*}
noting that $(\log M_2)^{r-1}\ll(\log X)^4$ is negligible. By the definition of $k$ and since $v_0\in I$, we have that 
\[
\frac{k}{v_0}\ll\frac{\log M_1}{U\log^2P}\ll 1.
\]
By the definition of $k$, we also have that
\[
k!
\ll(\log M_1)^{\frac{\log M_1}{\log P}}
\ll\exp\left(\frac{\log\log M_1 \log M_1}{(17+\varepsilon)\log\log X}\right)\\
\ll M_1^{\frac{1}{17+\varepsilon}}.
\]
Therefore, we can bound the integral over $\mathcal{T}_{2,\chi}$ by
\[
\sum_{\chi(q)}\int_{\mathcal{T}_{2,\chi}}|A(1+it,\chi)|^2dt
\ll\left(\varphi(q)\frac{T}{e^{kv_0/U}M_2}+\frac{\varphi(q)}{q}(2e^{k/U}-1)\right)X^{o(1)}M_1^{\frac{1}{17+\varepsilon}}.
\]
Since $v_0\in I$, we have that $e^{kv_0/U}\gg P^{k}\gg M_1$ by the definition of $k$. We also have that $2e^{k/U}-1\ll1$ and therefore we can bound the above integral by
\[
\sum_{\chi(q)}\int_{\mathcal{T}_{2,\chi}}|A(1+it,\chi)|^2dt
\ll\left(\varphi(q)\frac{T}{M_1M_2}+\frac{\varphi(q)}{q}\right) X^{o(1)}M_1^{\frac{1}{17+\varepsilon}}
\ll \varphi(q)X^{o(1)}M_1^{\frac{1}{17+\varepsilon}},
\]
as we have $M_1M_2=X^{1+o(1)}$ and $T\leq X^{1+o(1)}$. Returning to \eqref{eq:T2contribution}, we have the bound
\[
\sum_{\chi(q)}\int_{\mathcal{T}_{2,\chi}}|G_{v_0}(1+it,\chi)|^2|M_1(1+it,\chi)M_2(1+it,\chi)|^2dt
\ll \varphi(q)X^{o(1)}M_1^{\frac{1}{17+\varepsilon}+2(\alpha_1-\alpha_2)}P^{-2\alpha_1}.
\]
With our choices of $\alpha_1,\alpha_2$, we have that $2\alpha_1-2\alpha_2=-1/17$. Summing over $q$ introduces a factor of $Q_0^2$. Recalling that we may write $M_1=X^{\nu}$, for some $0<\nu\leq1/3+\varepsilon'$, we find that
\[
\sum_{q\leq Q_0}\frac{q}{\varphi(q)}\sum_{\chi(q)}\int_{\mathcal{T}_{2,\chi}}|G_{v_0}(1+it,\chi)|^2|M_1(1+it,\chi)M_2(1+it,\chi)|^2dt\ll Q_0^2P^{-2\alpha_1}X^{-\frac{\varepsilon\nu}{300}+o(1)},
\]
and choosing $\varepsilon'>0$ sufficiently small in terms of $\varepsilon>0$ ensures the above is bounded by $X^{-\varepsilon/1000}$, as needed.
\end{proof}

We now treat the contribution of the integral over $\mathcal{T}_\chi$, applying the Hal{\'a}sz-Montgomery inequality and the large value theorems.

\begin{lemma}\label{lem:Tintegral}
Let $\mathcal{T_\chi}$ be defined as in \eqref{eq:TT2defn}. Let $E>0$ be fixed sufficiently large. Then, we have that
\[
\sum_{q\leq Q_0}\frac{q}{\varphi(q)}\sum_{\chi(q)}\int_{\mathcal{T}_\chi}|G_{v_0}(1+it,\chi)M_1(1+it,\chi)M_2(1+it,\chi)|^2dt
\ll\frac{1}{\log^EX}.
\]
\end{lemma}

\begin{proof}
Let $M_1=X^{\nu}$ for some $0<\nu\leq1/3+\varepsilon'$. We first replace the integral over $\mathcal{T}_\chi$ with a sum over a well-spaced set. For each character $\chi$ mod $q$, cover $\mathcal{T}_\chi$ with intervals of unit length and from each interval take the point which maximises the integral over that interval. This set is not yet necessarily well-spaced, but we can split it into $O(1)$ well-spaced subsets. Therefore we may write
\begin{align*}
\sum_{\chi(q)}&\int_{\mathcal{T}_\chi}|G_{v_0}(1+it,\chi)M_1(1+it,\chi)M_2(1+it,\chi)|^2dt\\
&\ll\sum_{(t,\chi)\in\mathcal{T}'}|G_{v_0}(1+it,\chi)M_1(1+it,\chi)M_2(1+it,\chi)|^2,
\end{align*}
where $\mathcal{T}'$ is the well-spaced subset which maximises the right hand side. We now apply the prime-factored property ${|M_1(1+it,\chi)|^2\ll\log^{-F'}X}$ with $F'>0$ sufficiently large and then Lemma \ref{lem:HalaszMont} to get that 

\begin{align*}
&\sum_{(t,\chi)\in\mathcal{T}'}|G_{v_0}(1+it,\chi)M_1(1+it,\chi)M_2(1+it,\chi)|^2\\
&\ll\log^{-F'}X\sum_{(t,\chi)\in\mathcal{T}'}|G_{v_0}(1+it,\chi)M_2(1+it,\chi)|^2\\
&\ll\log^{-F'}X\left(\frac{\varphi(q)}{q}X^{1-\nu+\varepsilon'}+|\mathcal{T}'|(qT)^{1/2}\right)\sum_{e^{v_0/U}<p\leq e^{(v_0+1)/U}}\frac{1}{p^2}\sum_{M_2<m\leq2M_2}\frac{d_r^2(m)}{m^2}\\
&\ll\log^{-F}X\left(\frac{\varphi(q)}{q}+\frac{|\mathcal{T}'|(qT)^{1/2}}{X^{1-\nu+\varepsilon'}}\right),
\end{align*}
where $F>0$ is suitably large and $r\leq5$. If we can show that $|\mathcal{T}'|\ll X^{1/2-\nu-\varepsilon^2}$, then we will have that
\[
\sum_{\chi(q)}\int_{\mathcal{T}_\chi}|G_{v_0}(1+it,\chi)M_1(1+it,\chi)M_2(1+it,\chi)|^2dt
\ll\log^{-E'}X
\]
for some suitable $E'>0$. Summing over $q$, we have that
\[
\sum_{q\leq Q_0}\frac{q}{\varphi(q)}\sum_{\chi(q)}\int_{\mathcal{T}_\chi}|G_{v_0}(1+it,\chi)M_1(1+it,\chi)M_2(1+it,\chi)|^2dt
\ll\frac{1}{\log^EX},
\]
where $E>0$ is sufficiently large.

Thus, it remains to prove that $|\mathcal{T}'|\ll X^{1/2-\nu-\varepsilon^2}$.  Applying Lemma \ref{lem:Jutilalarge} with $M_1(1+it,\chi)^l, V=M_1^{-\alpha_2l},k=2,$ and $l\in\{2,3\}$, we have that
\begin{align*}
|\mathcal{T}'|&\ll
X^{\varepsilon^3}\left(M_1^{2\alpha_2l}+X^2M_1^{4l(2\alpha_2-1)}+XM_1^{4l(4\alpha_2-1)}\right)\\
&\ll
\begin{cases}
X^{\max(\frac{8}{17}\nu,2-\frac{104}{17}\nu,1-\frac{72}{17}\nu)-2\varepsilon^2},&\;(l=2),\\
X^{\max(\frac{12}{17}\nu,2-\frac{156}{17}\nu,1-\frac{108}{17}\nu)-2\varepsilon^2},&\;(l=3).
\end{cases}
\end{align*}
We have that $\nu\leq\frac{1}{3}+\varepsilon'$. The inequality $\frac{8}{17}\nu\leq\frac{1}{2}-\nu$ holds when $\nu\leq\frac{17}{50}$ and we have that $2-\frac{104}{17}\nu\geq\frac{8}{17}\nu$ when $\nu\leq\frac{17}{56}$. Note that $2-\frac{104}{17}\nu\leq\frac{1}{2}-\nu$ fails if $\nu<\frac{17}{58}$, so the inequality with $l=2$ provides the bound $|\mathcal{T}'|\ll X^{1/2-\nu-\varepsilon^2}$ in the range $\frac{17}{58}\leq\nu\leq\frac{1}{3}$. 

Similarly, $\frac{12}{17}\nu\leq\frac{1}{2}-\nu$ holds for $\nu\leq\frac{17}{58}$ and $2-\frac{156}{17}\nu\geq\frac{12}{17}\nu$ when $\nu\leq\frac{17}{84}$. We have that $2-\frac{156}{17}\nu\leq1/2-\nu$ fails when $\nu<\frac{51}{278}$, so the inequality with $l=3$ gives the required bound for $|\mathcal{T}'|$ when $\frac{51}{278}\leq\nu\leq\frac{17}{58}$.

For the remaining range $\nu<\frac{51}{278}$, we apply Lemma \ref{lem:largevalues} with $V=M_1^{-\alpha_2}$ to get that
\[
|\mathcal{T}'|
\ll (qT)^{2\alpha_2}X^{2\nu\alpha_2+\varepsilon}
\ll X^{\frac{4}{17}(1+\nu)+100\varepsilon}
\ll X^{1/2-\nu-\varepsilon^2},
\]
as required.
\end{proof}

\subsubsection{Type I Sums}
It remains to treat the contribution of the sums of form \eqref{eq:dirichletlsum}, applying the Cauchy-Schwarz inequality and a result on the twisted fourth moment of partial sums of Dirichlet $L$-functions.

\begin{lemma}\label{lem:dirichletcontribution}
Let $\varepsilon>0$ be fixed sufficiently small. With $N_1(s,\chi),N_2(s,\chi)$ as defined above, we have that
\[
\sum_{q\leq Q_0}\frac{q}{\varphi(q)}\sum_{\chi(q)}\int_{[T_0,T]\setminus\mathcal{T}_{1,\chi}}|G_{v_0}(1+it,\chi)|^2|N_1(1+it,\chi)N_2(1+it,\chi)|^2dt
\ll X^{-\varepsilon/2}.
\]
\end{lemma}

\begin{proof}
We split the domain of integration into dyadic intervals $[T_1,2T_1]$ such that $T_0\leq T_1\leq T$. As we are in the complement of $\mathcal{S}_1$, we have that $|G_{v_0}(1+it,\chi)P^{\alpha_1}|^{2(l-1)}\geq 1$, where we define $l=\lfloor\varepsilon\log X/\log P\rfloor$. Therefore, we have that 
\begin{align*}
\sum_{q\leq Q_0}&\frac{q}{\varphi(q)}\sum_{\chi(q)}\int_{([T_0,T]\setminus\mathcal{T}_{1,\chi})\cap[T_1,2T_1]}|G_{v_0}(1+it,\chi)|^2|N_1(1+it,\chi)N_2(1+it,\chi)|^2dt\\
\ll& P^{2\alpha_1(l-1)}\sum_{q\leq Q_0}\frac{q}{\varphi(q)}\sum_{\chi(q)}\int_{T_1}^{2T_1}|G_{v_0}(1+it,\chi)|^{2l}|N_1(1+it,\chi)N_2(1+it,\chi)|^2dt.
\end{align*}
Applying the Cauchy-Schwarz inequality three times (to the integral and the sums over $\chi$ and $q$), we have that the sum over $q$ above is bounded by
\begin{equation}\label{eq:dyadicintegral}
\ll \left(\sum_{q\leq Q_0}\frac{1}{\varphi(q)}\sum_{\chi(q)}\int_{T_1}^{2T_1}|G_{v_0}(1+it,\chi)|^{4l}|N_1(1+it,\chi)|^4dt\right)^{\frac{1}{2}}
\left(\sum_{q\leq Q_0}\frac{q^2}{\varphi(q)}\sum_{\chi(q)}\int_{T_1}^{2T_1}|N_2(1+it,\chi)|^4dt\right)^{\frac{1}{2}}.
\end{equation}
We apply Lemma \ref{lem:mvt} to the second integral. Noting that $N_2(1+it,\chi)$ either has coefficients $1$ or $\log n$, we find that
\[
\begin{aligned}
\sum_{q\leq Q_0}\frac{q^2}{\varphi(q)}\sum_{\chi(q)}\int_{T_1}^{2T_1}|N_2(1+it,\chi)|^4dt
&\ll X^{\varepsilon/10}\sum_{q\leq Q_0}q^2\left(\frac{T_1+N_2^2/q}{N_2^2}\right)\\
&\ll X^{\varepsilon/10}\left(\frac{T_1+N_2^2}{N_2^2}\right).
\end{aligned}
\]
We next treat the first term appearing in \eqref{eq:dyadicintegral}. We have that
\[
|G_{v_0}(1+it,\chi)|^{4l}
=\left|\sum_{e^{v_0/U}<p\leq e^{(v_0+1)/U}}\frac{\chi(p)}{p^{1+it}}\right|^{4l}
=\left|\sum_{e^{2lv_0/U}<n\leq e^{2l(v_0+1)/U}}\frac{\chi(n)a(n)}{n^{1+it}}\right|^2,
\]
where $a(n)=0$ unless $n$ is a product of $2l$ primes, not necessarily distinct, each lying in the interval $(e^{v_0/U},e^{(v_0+1)/U}]$. Writing $n$ in terms of its prime factorisation $n=p_1^{a_1}\cdots p_{b}^{a_{b}}$ with $b\leq 2l$, we have that $a(n)=\binom{2l}{a_1,\ldots,a_{b}}$ when it is non-zero and therefore that $a(n)\ll(2l)!$. Now we can apply Lemma \ref{lem:fourthmoment} with $M(1+it,\chi)=G_{v_0}^{2l}(1+it,\chi)$, $M=P^{2l}$ and $N$ corresponding to $N_1$ to get that
\begin{align*}
\sum_{q\leq Q_0}\frac{1}{\varphi(q)}\sum_{\chi(q)}\int_{T_1}^{2T_1}|G_{v_0}(1+it,\chi)|^{4l}|N_1(1+it,\chi)|^4dt
&\ll X^{\varepsilon/10}(2l)!^2\left(\frac{Q_0T_1}{N_1^2P^{2l}}\left(1+P^{4l}(Q_0T_1)^{-1/2}\right)+\frac{1}{T_1}\right)\\
&\ll X^{\varepsilon/10}(l!)^{4+\varepsilon}\left(\frac{Q_0T_1}{N_1^2P^{2l}}+\frac{1}{T_1}\right),
\end{align*}
as the definition of $l$ ensures that $P^{4l}(Q_0T_1)^{-1/2}\ll1$. Returning to \eqref{eq:dyadicintegral}, we have that
\begin{align*}
\sum_{q\leq Q_0}\frac{q}{\varphi(q)}&\sum_{\chi(q)}\int_{(\mathcal{T}\setminus\mathcal{T}_1)\cap[T_1,2T_1]}|G_{v_0}(1+it,\chi)|^2|N_1(1+it,\chi)N_2(1+it,\chi)|^2dt\\
&\ll P^{2\alpha_1(l-1)}X^{\varepsilon/10}(l!)^{2+\varepsilon}\left(\frac{Q_0T_1}{N_1^2P^{2l}}+\frac{1}{T_1}\right)^{1/2}\left(\frac{T_1+N_2^2}{N_2^2}\right)^{1/2}\\
&\ll P^{2\alpha_1(l-1)}X^{\varepsilon/10}(l!)^{2+\varepsilon}\left(\frac{Q_0T_1}{N_1^2N_2^2P^{2l}}(T_1+N_2^2)+\frac{1}{N_2^2}+\frac{1}{T_1}\right)^{1/2}.
\end{align*}
We have that $N_1N_2=X^{1+o(1)}$ with $N_1\geq X^{1/2-\varepsilon'}$ and $X^{1/3+\varepsilon'}\leq N_2\leq X^{1/2+\varepsilon'}$. As we also have that $X^{1/100}=T_0\leq T_1\leq T\leq X^{1+o(1)}$, the above is bounded by
\[
\ll P^{2\alpha_1(l-1)}X^{\varepsilon/10}(l!)^{2+\varepsilon}\left(\frac{1}{P^l}+\frac{1}{N_2}+\frac{1}{T_0^{1/2}}\right).
\]
Summing the contribution of each of the integrals over the dyadic intervals multiplies the above estimate by $\log X$. By the definition of $l$, we have that $(l!)^{2+\varepsilon}\ll(\log^2X)^{l(1+\varepsilon)}$, and so we also have that
\begin{align*}
 (P^{2\alpha_1-1}\log^2 X)^{l(1+\varepsilon)}
&\ll\exp\left((1+\varepsilon)\varepsilon\left(2\alpha_1-1+\frac{2}{17+\varepsilon}\right)\log X \right)\\
&\ll X^{-2\varepsilon/3}.
\end{align*}
Overall we have the bound
\[
\sum_{q\leq Q_0}\frac{q}{\varphi(q)}\sum_{\chi(q)}\int_{[T_0,T]\setminus\mathcal{T}_{1,\chi}}|G_{v_0}(1+it,\chi)|^2|N_1(1+it,\chi)N_2(1+it,\chi)|^2dt
\ll X^{-\varepsilon/2},
\]
as needed.
\end{proof}

\subsection{Completing the proof of Proposition \texorpdfstring{\ref{prop:Fmvtbound}}{6.5}}
We may now combine these estimates to complete the proof of Proposition \ref{prop:Fmvtbound}.

\begin{proof}[Proof of Proposition \ref{prop:Fmvtbound}.]
By Lemma \ref{lem:factorisedFLemma}, we have that
\begin{align*}
\sum_{q\leq Q_0}\frac{q}{\varphi(q)}\sum_{\chi(q)}\int_{T_0}^T |F(1+it,\chi)|^2dt
\ll \sum_{q\leq Q_0}\Bigg(&\frac{qU^2\log^2P}{\varphi(q)}\sum_{\chi(q)}\int_{T_0}^T|G_{v_0}(1+it,\chi)|^2|H_{v_0}(1+it,\chi)|^2dt\\
&+\frac{1}{Q_0\log^{2+\eta}X}\left(\frac{qT\log X}{X}+\frac{q}{\varphi(q)}\right)\Bigg)\\
\ll \sum_{q\leq Q_0}\Bigg(&\frac{q U^2\log^2P}{\varphi(q)}\sum_{\chi(q)}\Bigg(\int_{\mathcal{T}_{1,\chi}}|G_{v_0}(1+it,\chi)|^2|H_{v_0}(1+it,\chi)|^2dt\\
&+\int_{[T_0,T]\setminus\mathcal{T}_{1,\chi}}|G_{v_0}(1+it,\chi)|^2|H_{v_0}(1+it,\chi)|^2dt\Bigg)\\
&+\frac{1}{Q_0\log^{2+\eta}X}\left(\frac{qT\log X}{X}+\frac{q}{\varphi(q)}\right)\Bigg)
\end{align*}
for some $\eta=\eta(\varepsilon)>0$ and some suitable integer $v_0\in I$. We apply Lemma \ref{lem:T1integral} to bound the contribution of the integral over $\mathcal{T}_{1,\chi}$, finding that the above is bounded by
\begin{align*}
\ll \sum_{q\leq Q_0} \Bigg(\frac{qU^2\log^2P}{\varphi(q)}\sum_{\chi(q)}\int_{[T_0,T]\setminus\mathcal{T}_{1,\chi}} & |G_{v_0}(1+it,\chi)|^2|H_{v_0}(1+it,\chi)|^2dt\\
&+\frac{1}{Q_0\log^{2+\eta}X}\left(\frac{qTP\log X}{X}+\frac{q}{\varphi(q)}\right)\Bigg).
\end{align*}
Combining Lemmas \ref{lem:T2integral} to \ref{lem:dirichletcontribution},  we bound the contribution of the complement of $\mathcal{S}_1$ by
\[
\ll \frac{U^2\log^2P}{\log^EX}
\ll\frac{1}{\log^FX}
\]
for some sufficiently large $F>0$, which is negligible. Thus, we have that
\[
\sum_{q\leq Q_0}\frac{q}{\varphi(q)}\sum_{\chi(q)}\int_{T_0}^T|F(1+it,\chi)|^2dt
\ll \frac{1}{Q_0\log^{2+\eta}X} \sum_{q\leq Q_0} \left(\frac{qTP\log X}{X}+\frac{q}{\varphi(q)}\right),
\]
as required.
\end{proof}

\section{Proof of Theorem \texorpdfstring{\ref{thm:unrestricted}}{1.3}}
\label{sec:unrestricted}

We now briefly outline how to adjust the argument to prove Theorem \ref{thm:unrestricted}. The problem can be reduced to the set of $E_2$ numbers which factorise in the ``typical" way. By Mertens' theorem, almost all products of exactly two primes $p_1p_2\leq X$ with $p_1\leq p_2$ satisfy 
\begin{equation}\label{eq:typicalap}
p_1\in\left[\exp\left((\log X)^{\varepsilon(X)}\right),\exp\left((\log X)^{1-\varepsilon(X)}\right)\right]=:[P_1,P_2],
\end{equation}
where $\varepsilon(X)=o(1)$. We define $E''_2:=E''_2(X)$ to be the set of $E_2$ numbers $n=p_1p_2\in(X,2X]$ which factorise in the typical way. Using a sieve theory argument, we have that
\[
\begin{aligned}
\frac{1}{X}\sum_{X<n\le2X}\mathbbm{1}_{E_2}(n)\mathbbm{1}_{E_2}(n+h)-o\left(\frac{\mathfrak{S}(h)(\log\log X)^2}{(\log X)^2}\right)&\le \frac{1}{X}\sum_{X<n\le 2X}\mathbbm{1}_{E''_2}(n)\mathbbm{1}_{E''_2}(n+h)\\
&\le \frac{1}{X}\sum_{X<n\le2X}\mathbbm{1}_{E_2}(n)\mathbbm{1}_{E_2}(n+h).
\end{aligned}
\]
Therefore, we can reduce the problem to considering the correlations of $n,n+h\in E''_2$.

We modify every definition featuring $(P,P^{1+\delta}]$, replacing this interval with $[P_1,P_2]$. We will once again apply the Hardy-Littlewood circle method and in \eqref{eq:majordef} and \eqref{eq:minordef} we take
\begin{equation}\label{eq:unrestrictedparam}
Q_0:=\log^{A'} X, A'>4,\qquad Q:=P_2\log^CX,\qquad H\geq Q\log^DX,
\end{equation}
where $C$ is chosen sufficiently large in terms of $A'$ and $D$ is chosen sufficiently large in terms of $A'$ and $C$. In Lemma \ref{lem:factorisedFLemma} we instead define $I:=[U\log P_1,U\log P_2]$ where $U:=Q_0^{E}$, $E>0$ and we define $\alpha_1:=\varepsilon'>0$ sufficiently small in terms of $\varepsilon>0$. 

The applications of the Cauchy-Schwarz inequality to sums over products of exactly two primes in the proofs of Proposition \ref{prop:minorestimate} and Proposition \ref{prop:B2Xbound} are now too inefficient. To overcome this, we split the sum over $p_1\in[P_1,P_2]$ into dyadic intervals before applying the inequality. We now outline how to modify the proof of Proposition \ref{prop:minorestimate}.

\begin{prop}\label{prop:unrestrictedminors}
Let $A>3$, $B>1$ be fixed and $\mathfrak{m}$ be defined as in \eqref{eq:minordef} with $Q_0,Q$ as in \eqref{eq:unrestrictedparam}. Let $Q\log^DX\leq H\leq X\log^{-A}X$ with $D>0$ sufficiently large.  For $\alpha\in\mathfrak{m}$ we have that
\[
\int_{\mathfrak{m}\cap[\alpha-\frac{1}{2H},\alpha+\frac{1}{2H}]}|S(\theta)|^2d\theta
\ll\frac{X}{\log^BX}.
\]
\end{prop}

\begin{proof}
As before, we apply Lemma \ref{lem:gallagherslemma} to the minor arc integral so that we need to bound
\[
I\ll\frac{1}{H^2}\int_{X}^{2X}\left|\sum_{x<n\leq x+H}\varpi_2(n)e(n\alpha)\right|^2dx+H\log^2 X.
\]
The second term is acceptable by our choice of $H$, so it remains to bound the first term. Now before applying Cauchy-Schwarz to the integrand we split the sum over $p_1$ into dyadic intervals $[P,2P]$ with $P_1\leq P\leq P_2$ so that we instead need to integrate
\[
\Bigg(\sum_{P<m_1\leq 2P}|\primecf(m_1)|^2\Bigg)\Bigg(\sum_{P<m_2\leq 2P}\Bigg|\sum_{x<m_2p\leq x+H}(\log p)e(\alpha m_2 p)\Bigg|^2\Bigg).
\]
The first term is $\ll \frac{P}{\log P}$, while the second term is equal to
\[
\sum_{\substack{x<mp_1,mp_2\leq x+H \\ P<m\leq 2P}} (\log p_1)(\log p_2) e(\alpha m(p_1-p_2)).
\]
Next, we perform the integration on this sum and split into the diagonal ($p_1=p_2$) and off-diagonal terms ($p_1\neq p_2$) as before. The diagonal terms now contribute
\begin{equation}\label{eq:unrestricteddiag}
S_1\ll \frac{P}{H\log P}\sum_{P<m\leq 2P}\sum_{\frac{X}{m}<p\leq\frac{3X}{m}}\log^2 p
\ll \frac{XP\log X}{H\log P}
\ll\frac{X}{\log^{C+D-1}X}.
\end{equation}
Once again applying Lemma \ref{lem:expsumbound} followed by Lemma \ref{lem:primepairs} and Lemma \ref{lem:distancesum}, the off-diagonal terms contribute
\begin{equation}\label{eq:unrestrictedoffdiag}
S_2\ll\frac{X\log\log X \log X}{\log P}\left(\frac{1}{Q_0}+\frac{1}{P}+\frac{Q}{H}\right)
\ll\frac{X}{\log^{B'}X}
\end{equation}
for $B'>3$ by our choice of $Q_0,P,Q$ and $H$. Combining the contributions of the dyadic intervals $[P,2P]$ gives that
\[
I\ll\frac{X}{\log^BX}
\]
for $B>1$, as claimed.
\end{proof}

\begin{prop}\label{prop:unrestrictedmajor}
Let $A>3,B>0$ be fixed. Let $\varepsilon>0$ be fixed and $\exp((\log X)^{1-\varepsilon})\leq H\leq X\log^{-A}X$. Let $\mathfrak{M}$ be defined as in \eqref{eq:majordef} with $Q_0,Q$ as in \eqref{eq:unrestrictedparam}. Then, for all but at most $O(HQ_0^{-1/3})$ values of $0<|h|\leq H$ we have that
\[
\int_\mathfrak{M} |S(\alpha)|^2e(-h\alpha) d\alpha=\mathfrak{S}(h)X\left(\sum_{P_1\leq p\leq P_2}\frac{1}{p}\right)^2+ O\left(\frac{X}{\log^B X}\right).
\]
\end{prop}

\begin{proof}
Recalling Lemma \ref{lem:ExpandS}, we have the expansion
\begin{equation}\label{eq:Sabexpand}
\begin{aligned}
S(\alpha)=&\frac{\mu(q)}{\varphi(q)}\sum_{P_1<p\leq P_2}\frac{1}{p}\sum_{X<n\leq2X}e(\beta n)\\
&+\frac{1}{\varphi(q)}\sum_{\chi(q)}\tau(\overline{\chi})\chi(a)\sum_{X<n\leq2X}\left(\varpi_2(n)\chi(n)-\delta_\chi\sum_{P_1\leq p\leq P_2}\frac{1}{p}\right)e(\beta n)\\
=&a(\alpha)+b(\alpha)
\end{aligned}
\end{equation}
and following the argument of Section \ref{sec:Majors} we have that
\[
\int_\mathfrak{M} |S(\alpha)|^2e(-h\alpha) d\alpha=\mathfrak{S}(h)X\left(\sum_{P_1\leq p\leq P_2}\frac{1}{p}\right)^2+O\left(\frac{X}{\log^BX}+A(X)B(X)+B^2(X)\right).
\]
Note that $A^2(X)\ll X(\log\log X)^3$, so it remains to bound $B^2(X)$. Following Proposition \ref{prop:B2decomposition}, we have that $B^2(X)\ll B_1(X)+B_2(X)$ where $B_1(X)$ is defined in \eqref{eq:B1def} and $B_2(X)$ is defined in \eqref{eq:B2def} with $(P,P^{1+\delta}]$ replaced with $[P_1,P_2]$. With our choices of \eqref{eq:unrestrictedparam} and $P_1,P_2$, following the arguments of Proposition \ref{prop:B1Xbound} and Proposition \ref{prop:Fmvtbound} we now have that
\[
B_1(X)\ll\frac{X}{\log ^BX}.
\]

The proof of Proposition \ref{prop:B2Xbound} requires modifying in a similar way to Proposition \ref{prop:minorestimate}. In \eqref{eq:B2CS} we split the sum over $P_1\leq p_1\leq P_2$ into dyadic intervals $P<p_1\leq 2P$ before applying Cauchy-Schwarz, Lemma \ref{lem:primesalmostall} and then combining the contributions of the dyadic sums.
\end{proof}

We are now able to complete the proof of Theorem \ref{thm:unrestricted}.

\begin{proof}[Proof of Theorem \ref{thm:unrestricted}.]
By partial summation and Mertens' theorem we have the bound
\[
\begin{aligned}
\int_0^1|S(\alpha)|^2d\alpha
&=\sum_{X<n\leq2X}\varpi_2^2(n)
\ll \log X\sum_{X<n\leq2X}\varpi_2(n)
\ll X\log X\sum_{P_1\leq p\leq P_2}\frac{1}{p}\\
&\ll X\log X \log\log X.
\end{aligned}
\]
Therefore, following the proof of Theorem \ref{thm:mainthm}, the result may be deduced from combining this bound with Proposition \ref{prop:unrestrictedminors}, an application of Chebyshev's inequality and Proposition \ref{prop:unrestrictedmajor} followed by an application of partial summation.
\end{proof}

\section{Proof of Theorem \texorpdfstring{\ref{thm:primealmost}}{1.4}}
\label{sec:primealmost}

We outline the modifications needed to prove Theorem \ref{thm:primealmost}. When applying the Hardy-Littlewood circle method, in \eqref{eq:majordef} and \eqref{eq:minordef} we now choose 
\begin{equation}\label{eq:papparam}
Q_0:=\log^{A'}X, A'>6, \qquad Q:=X^{1/6+\varepsilon/2},\qquad H\geq QX^{\varepsilon/2}.
\end{equation}
As in Section \ref{sec:unrestricted},  in Lemma \ref{lem:factorisedFLemma} we instead define $I:=[U\log P_1,U\log P_2]$ where $U:=Q_0^{E}$, $E>0$ and we define $\alpha_1:=\varepsilon'>0$ sufficiently small in terms of $\varepsilon>0$. Analogously to the almost prime case, we may write
\[
\sum_{X<n\leq2X}\Lambda(n)\varpi_2(n+h)=\int_0^1 S(\alpha)\overline{S'(\alpha)}e(-h\alpha)d\alpha+O(h\log^2X),
\]
where for $\alpha\in(0,1)$ we define the exponential sum $S'(\alpha):=\sum_{X<n\leq 2X}\Lambda(n)e(n\alpha)$. The error term is acceptable by our choice of $H$.  We have the following result for the major arcs.

\begin{prop}\label{prop:papmajor}
Let $A>5,B>0$ be fixed and let $\varepsilon>0$ be fixed sufficiently small. Let $X^{1/6+\varepsilon}\leq H\leq X\log^{-A}X$. Let $\mathfrak{M}$ be defined as in \eqref{eq:majordef} with $Q_0,Q$ as in \eqref{eq:papparam}. Then, for all but at most $O(HQ_0^{-1/3})$ values of $0<|h|\leq H$ we have that
\[
\int_\mathfrak{M} S(\alpha)\overline{S'(\alpha)}e(-h\alpha) d\alpha=\mathfrak{S}(h)X\left(\sum_{P_1\leq p\leq P_2}\frac{1}{p}\right)+ O\left(\frac{X}{\log^B X}\right).
\]
\end{prop}

\begin{proof}
We can expand $S'$ in terms of Dirichlet characters (see for example \cite{mikawa}):
\begin{align*}
S'(\alpha)=&\frac{\mu(q)}{\varphi(q)}\sum_{X<n\leq2X}e(n\beta)+\frac{1}{\varphi(q)}\sum_{\chi(q)}\tau(\bar{\chi})\chi(a)\sum_{X<n\leq2X}(\Lambda(n)\chi(n)-\delta_\chi)e(n\beta)
+O(\log^2X)\\
=&c(\alpha)+d(\alpha)+O(\log^2X),
\end{align*}
say. Therefore, using the expansion \eqref{eq:Sabexpand} and Cauchy-Schwarz, we may write the integral over the major arcs as
\begin{equation}\label{eq:primeapexpand}
\begin{aligned}
\int_{\mathfrak{M}}S(\alpha)\overline{S'(\alpha)}e(-h\alpha)d\alpha
=&\int_\mathfrak{M}a(\alpha)\overline{c(\alpha)}e(-h\alpha)d\alpha\\
&+O\big(A(X)D(X)+B(X)(C(X)+D(X))+(A(X)+B(X))\log^2X\big),
\end{aligned}
\end{equation}
where we define $C^2(X)=\int_\mathfrak{M}|c(\alpha)|^2d\alpha$ with $D(X)$ defined analogously. Evaluating $\int_{\mathfrak{M}}a(\alpha)\overline{c(\alpha)}e(-h\alpha)d\alpha$ as in Section 5 gives the required main term and an acceptable error. Mikawa \cite[Section 3]{mikawa} proves that $C^2(X)\ll X\log\log X$ and that
\[
D^2(X)\ll\sum_{q\leq Q_0}\frac{q}{\varphi(q)}\sum_{\chi(q)}\Bigg(\int_X^{2X}\Bigg|\frac{1}{qQ}\sum_{x<n\leq x+qQ/2}(\Lambda(n)\chi(n)-\delta_\chi)\Bigg|^2dx+qQ\log^2X\Bigg).
\] 
The second term is negligible by the definition of $Q$. Noting that we have chosen $Q=X^{1/6+\varepsilon/2}$, we apply Lemma \ref{lem:primesalmostall} to the first term to get 
\[
D^2(X)
\ll\frac{X}{\log^BX}
\]
for $B>0$, as required. Combining this with our estimates for $A(X),B(X)$ (from Proposition \ref{prop:unrestrictedmajor}) and $C(X)$ we have that the error term in \eqref{eq:primeapexpand} is $O(X\log^{-B}X)$, as required.
\end{proof}

\begin{proof}[Proof of Theorem \ref{thm:primealmost}.]
Analogously to the proof of Theorem \ref{thm:weightedmainthm}, by \cite[Proposition 3.1]{mrt} we have that
\begin{align*}
\sum_{0<|h|\leq H}&\left|\sum_{X<n\leq2X}\Lambda(n)\varpi_2(n+h)-\int_{\mathfrak{M}}S(\alpha)\overline{S'(\alpha)}e(-h\alpha)d\alpha\right|^2\\
&\ll H\int_\mathfrak{m}|S(\alpha)||S'(\alpha)|\int_{\mathfrak{m}\cap[\alpha-\frac{1}{2H},\alpha+\frac{1}{2H}]}|S(\beta)||S'(\beta)|d\beta d\alpha.
\end{align*}
By Cauchy-Schwarz, we have that the above is bounded by
\[
\ll H\left(\int_0^1|S'(\alpha)|^2d\alpha\right)\left(\int_0^1|S(\alpha)|^2d\alpha\right)^{1/2}\left(\sup_{\alpha\in\mathfrak{m}}\int_{\mathfrak{m}\cap[\alpha-\frac{1}{2H},\alpha+\frac{1}{2H}]}|S(\beta)|^2d\beta\right)^{1/2}.
\]
Trivially, we have that
\[
\int_0^1|S(\alpha)|^2d\alpha
\ll X\log X\log\log X,\qquad
\int_0^1|S'(\alpha)|^2d\alpha
\ll X\log X,
\]
so, combining these estimates with Proposition \ref{prop:unrestrictedminors} (suitably adjusting for the choices of $Q_0,Q,H$), we have that
\[
\sum_{0<|h|\leq H}\left|\sum_{X<n\leq2X}\Lambda(n)\varpi_2(n+h)-\int_{\mathfrak{M}}S(\alpha)\overline{S'(\alpha)}e(-h\alpha)d\alpha\right|^2\ll\frac{HX^2}{\log^BX}.
\]
Therefore, applying Chebyshev's inequality and Proposition \ref{prop:papmajor} followed by partial summation gives the result.
\end{proof}

\section*{Acknowledgements}
The author is grateful to her supervisor Stephen Lester for suggesting the problem and for many helpful comments and discussions throughout this work. The author would like to thank Joni Ter\"av\"ainen for useful comments on an earlier draft of this article, which included Theorem \ref{thm:unrestricted} and its proof and an improvement on the log exponent in Theorem \ref{thm:mainthm}. The author would also like to thank Kaisa Matom\"aki for helpful comments on an earlier draft. She thanks the anonymous referee for a careful reading of the paper and useful comments. This work was supported by the Engineering and Physical Sciences Research Council [EP/R513106/1].

\bibliographystyle{abbrv}
\bibliography{AlmostPrimesBib}

\end{document}